\documentclass[11pt,a4paper]{article}
\usepackage{amssymb,amsmath,amsopn,amsthm,graphicx,makeidx, url}
\input xy
\xyoption{all}
\usepackage[all,2cell]{xy}
\UseAllTwocells

\begin{document}

\newtheorem{definition}{Definition}[section]
\newtheorem{definitions}[definition]{Definitions}
\newtheorem{lemma}[definition]{Lemma}
\newtheorem{prop}[definition]{Proposition}
\newtheorem{theorem}[definition]{Theorem}
\newtheorem{cor}[definition]{Corollary}
\newtheorem{cors}[definition]{Corollaries}
\theoremstyle{remark}
\newtheorem{remark}[definition]{Remark}
\theoremstyle{remark}
\newtheorem{remarks}[definition]{Remarks}
\theoremstyle{remark}
\newtheorem{notation}[definition]{Notation}
\theoremstyle{remark}
\newtheorem{example}[definition]{Example}
\theoremstyle{remark}
\newtheorem{examples}[definition]{Examples}
\theoremstyle{remark}
\newtheorem{dgram}[definition]{Diagram}
\theoremstyle{remark}
\newtheorem{fact}[definition]{Fact}
\theoremstyle{remark}
\newtheorem{illust}[definition]{Illustration}
\theoremstyle{remark}
\newtheorem{rmk}[definition]{Remark}
\theoremstyle{definition}
\newtheorem{observation}[definition]{Observation}
\theoremstyle{definition}
\newtheorem{question}[definition]{Question}
\theoremstyle{definition}
\newtheorem{conj}[definition]{Conjecture}

\newcommand{\stac}[2]{\genfrac{}{}{0pt}{}{#1}{#2}}
\newcommand{\stacc}[3]{\stac{\stac{\stac{}{#1}}{#2}}{\stac{}{#3}}}
\newcommand{\staccc}[4]{\stac{\stac{#1}{#2}}{\stac{#3}{#4}}}
\newcommand{\stacccc}[5]{\stac{\stacc{#1}{#2}{#3}}{\stac{#4}{#5}}}

\renewcommand{\marginpar}[2][]{}

\renewenvironment{proof}{\noindent {\bf{Proof.}}}{\hspace*{3mm}{$\Box$}{\vspace{9pt}}}

\title{Strictly atomic modules in definable categories}

\author{Mike Prest \\ Department of Mathematics, University of Manchester, UK \\ mprest@manchester.ac.uk}

\date{\today} 

\maketitle


\abstract{If ${\cal D}$ is a definable category then it may contain no nonzero finitely presented objects but, by a result of Makkai, there is a $\varinjlim$-generating set of strictly ${\cal D}$-atomic modules in ${\cal D}$.  These modules share some key properties with finitely presented modules.

We consider these modules in general and then in the case that ${\cal D}$ is the category of modules of some fixed irrational slope over a tubular algebra.}

\tableofcontents

\section{Introduction}

Mittag-Leffler and strictly Mittag-Leffler modules were introduced in \cite{RayGru}.  These modules are in some sense `small':  they include the finitely presented modules and pure-projective modules (direct summands of direct sums of finitely presented modules).  For countably generated modules, the conditions of being Mittag-Leffler, strictly Mittag-Leffler and pure-projective are equivalent.  

Definable categories include, but are much more general than, module categories.  They are not in general locally finitely presented; indeed they may contain no finitely presented objects other than $0$ (\cite[18.1.1]{PreNBK}). They do, however, have enough relative (to the definable category) Mittag-Leffler, even strictly Mittag-Leffler, objects; that is a result of Makkai \cite{Makk}, \cite{MakkTop}.  Here we deduce various consequences.  We also give a proof of existence, in the case of modules over countable rings, which is more direct than in Makkai's paper.  We also favour a different terminology for the relative concepts, using the term (strictly) ${\cal D}$-atomic for the relativisation of (strictly) Mittag-Leffler to a definable category ${\cal D}$.  This terminology reflects the characterisation of Mittag-Leffler modules which is that every finite tuple of elements in such a module has finitely generated pp-type.

Early papers dealing with these modules include \cite{Azu}, where the strict Mittag-Leffler condition was shown to be equivalent to being locally pure-projective, and \cite{RotHab}, where the model-theoretic characterisation of Mittag-Leffler modules in terms of pp-types was discovered and the basic results extended to definable categories.  Makkai's work \cite{MakkTop}, \cite{Makk} was done in a very general context using category-theory-inspired model theory and here we connect it with the more algebraic line of development.

After introducing the concepts and basic results, we recall Makkai's result - existence of enough strictly ${\cal D}$-atomic objects in every definable category ${\cal D}$ - and we give a fairly short direct proof in the case that ${\cal D}$ is defined over a countable ring.  Makkai's result implies, for example, that, if ${\cal D}$ is a definable subcategory of a module category ${\rm Mod}\mbox{-}R$, then every finitely presented $R$-module has a strictly ${\cal D}$-atomic ${\cal D}$-preenvelope.  Then we look at some immediate consequences, including the case that ${\cal D} = {\rm Gen}(T)$ for a silting $R$-module $T$.

All this is applied in the category ${\cal D}_r$ of $R$-modules of some irrational slope $r$ when $R$ is a tubular algebra.  These are tame, generally non-domestic, algebras; their categories of finite-dimensional modules are described in \cite[Chpt. 5]{RinTame} and a feature is that every finite-dimensional indecomposable has a slope (a rational number or $\infty$).  Moreover, if there is a non-zero morphism from a module of slope $r$ to one of slope $s$, then $r\leq s$.  It was shown by Reiten and Ringel \cite[13.1]{ReiRin} that, remarkably, {\em every} indecomposable module has a slope, which is a real number or $\infty$.  If $r$ is irrational, then the category ${\cal D}_r$ of modules of slope $r$ contains no finite-dimensional nonzero module.  A good deal of information has been obtained about these categories in \cite{HarPre}, \cite{GreTub}, \cite{AngKusAlg} and \cite{KusLak} but there is currently no description of the indecomposable pure-injectives in ${\cal D}_r$.  Here, with that aim in mind, we shed a little more light on the structure of ${\cal D}_r$.

\vspace{4pt}

In somewhat more detail, Section \ref{secML} brings together background definitions and results.  Strictly ${\cal D}$-atomic modules, which coincide, \ref{streqstr}, with the strict ${\cal D}$-stationary modules from \cite[8.2]{AngHerML} in ${\cal D}$, are given an internal-to-${\cal D}$-characterisation in \ref{charsML} and the indecomposable direct summands of their character-duals are shown to be all neg-isolated, \ref{dualnegisol}, cf.~\cite{AngHerzLak}.  Then Makkai's result for countable rings is given a direct proof, \ref{Mak}.  In Section \ref{secsatgen}, we look at strictly ${\cal D}$-atomic modules as `pure generators' for ${\cal D}$.  In \ref{rdsbiend} we show that, if $M$ is a strictly ${\cal D}$-atomic module which is finitely generated over its endomorphism ring, then the ring of definable scalars of $M$ is its biendomorphism ring.  We note in Section \ref{sectilt} that a silting module is strictly atomic for the silting class that it generates.

In Section \ref{secirrat} we specialise to definable categories of the form ${\cal D}_r$.  We see that if $T\in {\cal D}_r$ is a tilting module, then the strictly atomic modules in ${\cal D}_r$ are the direct summands of direct sums of copies of $T$, \ref{strDrat}.  Every exact sequence in ${\cal D}_r$ is pure-exact, \ref{espes}, and we want to understand the non-pure morphisms in ${\cal D}_r$, in particular those with kernel not in ${\cal D}_r$ (see the proof of \cite[6.5]{AngKusAlg} for such morphisms).  In \ref{kerinDr} we characterise the submodules $K$ of $D\in {\cal D}_r$ such that $D/K \in {\cal D}_r$ as those which are definably closed in $D$ and we develop some consequences.

With the aim of making the paper reasonably self-contained, Section \ref{secmodth} gives a quick account of the ideas, such as pp-types and definable closure, from model theory that we use in the paper.

\vspace{4pt}

I would like to thank Amit Kuber, for reminding me about Makkai's paper \cite{Makk} and the strength of his result and for subsequent discussions, and Philipp Rothmaler for a conversation where we worked out the more direct proof, in the countable case, of Makkai's result and, of course, for originally noticing and developing in \cite{RotHab} the model-theoretic content and import of the Mittag-Leffler condition.  Thanks to Rosanna Laking for pointing out an example showing that some restriction on the duality is needed for \ref{dualnegisol}.  Thanks also to Lidia Angeleri H\"{u}gel for clarifying the connection with strict stationary modules and for strengthening the conclusion of \ref{strDrat} from pure to split embeddings.  I also thank the referees for a number of useful comments and suggestions.

\section{Background}

Throughout we use the language of rings and modules but, in fact, for the general results we may take $R$ to be any (skeletally) small preadditive category - a ring with many objects - and so left and right $R$-modules will be additive functors (respectively covariant and contravariant) from $R$ to the category ${\bf Ab}$ of abelian groups.  For this paper, we don't need that generality so, throughout, we write as if $R$ is a normal, 1-sorted, ring but the proofs do work in the more general context.

We also make full use of concepts and results from the model theory of modules as well as algebraic methods.  Indeed, from the start, we freely use the notions of pp formula and pp-type.  There is a section at the end of the paper which, we hope, explains what is needed here.  There are many sources for more detail about the model theory used:  here I tend to cite \cite{PreNBK} as a fairly comprehensive secondary source but there are numerous (much) more concise introductions and summaries.  In particular there is \cite{RotML2} which also includes a great deal of the algebraic background material from Sections \ref{secML} and \ref{secSML}, for which also see \cite{AngHerML}.  The ``additive model theory" that we use is really a (highly-developed) part of regular (=pp-) model theory; see \cite{Butz} for an introduction to regular model theory which is based on categorical model theory rather than classical model theory.

\section{Mittag-Leffler and relatively atomic modules}\label{secML} \marginpar{secML}

Let $R$ be a ring.  Throughout ${\rm Mod}\mbox{-}R$ and $R\mbox{-}{\rm Mod}$ denote, respectively, the categories of right and of left $R$-modules; ${\rm mod}\mbox{-}R$ is the category of finitely presented right modules.

An $R$-module $M$ is {\bf Mittag-Leffler}, or just {\bf ML} (\cite[\S 2]{RayGru}), if $M$ is the direct limit of a directed system $(\{M_i\}_i, \{f_{ij}:M_i \to M_j\}_{i\leq j})$ of finitely presented modules $M_i$, where the directed system satisfies the following equivalent conditions, with $f_{i\infty}:M_i \to M$ denoting the limit maps:

\noindent (i) for every $i$ there is $j\geq i$ such that, for any tuple $\overline{a}$ from $M_i$, ${\rm pp}^M(f_{i \infty}\overline{a}) = {\rm pp}^{M_j}(f_{ij}\overline{a})$ (it is enough to require this for a generating tuple for $M_i$);

\vspace{4pt}

\noindent (ii) for every $i$ there is $j\geq i$ such that $f_{ij}$ factors through each $f_{ik}$ for all $k \geq j$.

\vspace{4pt}

\noindent For the pp-types, ${\rm pp}^M(-)$ referred to in (i), see Section \ref{secmodth}.

\begin{theorem} \label{MLchar} \marginpar{MLchar}  Suppose that $M$ is a right $R$-module.  Then the following conditions are equivalent.

\noindent (i) $M$ is Mittag-Leffler.

\noindent (ii) For every set $\{ L_i\}_{i\in I}$ of left $R$-modules, the canonical map $M\otimes_R \, (\prod_{i\in I} \, L_i) \to \prod_{i\in I} \, (M\otimes_R L_i)$ is monic.

\noindent (iii)  Every pp-type realised in $M$ is finitely generated.  That is, for any $\overline{a} = (a_1,\dots, a_n)$ with the $a_i\in M$, ${\rm pp}^M(\overline{a}) = \langle \phi\rangle$ for some pp formula $\phi$, where $\langle \phi\rangle =\{ \psi \mbox{ pp : } \phi \leq \psi\}$ denotes the pp-type generated by $\phi$.
\end{theorem}

The equivalence of (i) and (iii) is a special case of \cite[2.2]{RotHab}.  It is the property (iii) which we will use in what follows as the definition of the corresponding modules in the more general context of definable categories.  The above result, done for general definable categories ${\cal D}$ in place of ${\rm Mod}\mbox{-}R$, is in \cite{RotHab} (e.g.~\cite[2.2]{RotHab}); also see subsequent papers, e.g.~\cite{RotML}, ~\cite{PunRoth},
\cite{AngHerML}, \cite{GAIRT}.  In the terminology of those references we would refer to our modules of interest as ${\cal D}^{\rm d}$-Mittag-Leffler, where ${\cal D}^{\rm d}$ is the dual definable category (see Section \ref{secmodth}) of ${\cal D}$.  This would be slightly clumsy and also not extendable to non-additive definable categories (\cite{KubRo}, \cite{LackTend}) where there seems not to be the nice, multiple-level, theory of duality that one has in the additive context.  So we will follow \cite{RotSML} (and the earlier references \cite{RotHab}, \cite{KucRot} and \cite{RotML2}) and take property (iii) as the basis of our terminology.

We recall some properties of Mittag-Leffler modules and their relation to {\bf pure-projective} modules - modules which are direct summands of (infinite) direct sums of finitely presented modules.

\begin{theorem}

\noindent (1) Every pure-projective $R$-module, in particular every finitely presented $R$-module, is ML.

\noindent (2) Every direct sum of ML modules is ML, as is every pure submodule, in particular every direct summand, of an ML module.

\noindent (3) Every countably generated ML module is pure-projective.

\noindent (4) A module $M$ is ML iff every finite subset of $M$ is contained in a pure-projective pure submodule of $M$, and this implies that every countable subset of $M$ is contained in a pure-projective pure submodule of $M$.
\end{theorem}

An embedding $j:A \to B$ of modules is {\bf pure} if, for every left module $L$, $j\otimes_RL: A\otimes_RL \to B\otimes_RL$ is monic, equivalently if, for every finite tuple, $\overline{a}$ of elements from $A$ and every pp formula $\phi$, if $j\overline{a} \in \phi(B)$, then $\overline{a} \in \phi(A)$ (again, see Section \ref{secmodth} for any unfamiliar notation or terminology).

\vspace{4pt}

\noindent {\bf Convention:}  Throughout the paper ${\cal D}$ will denote a definable additive category which we will take to be definably embedded in (i.e.~equivalent to a definable subcategory of) ${\rm Mod}\mbox{-}R$ for some small preadditive category $R$.  All references to pp formulas and types may be taken to refer to the language for $R$-modules.

If ${\cal D}$ is definably embedded in ${\rm Mod}\mbox{-}R$ (indeed, definably embedded into any definable category ${\cal E}$) then (\cite[5.3]{PreMAMS}) purity in the larger category, restricted to ${\cal D}$ coincides with the internally-defined purity in ${\cal D}$.  The latter is defined as follows:  a morphism in ${\cal D}$ is a {\bf pure embedding} iff some ultrapower of it is a split embedding.  This works because, given any definable category ${\cal D}$, there is some index set $I$ and ultrafilter ${\cal U}$ on that index set such that, if $D\in {\cal D}$, then the ultrapower $D^\ast = D^I/{\cal U}$ is pure-injective \cite[21.2]{PreMAMS}.  That is by a general model-theoretic result; see \cite[\S\S 20,21]{PreMAMS} for more detail.

It follows that the choice of category of modules into which ${\cal D}$ is definably embedded makes no difference to the model theory on ${\cal D}$ (only the language might change to one that is, when restricted to ${\cal D}$, equivalent).

\begin{definition} Given a definable category ${\cal D}$ and pp formulas $\phi, \psi$ (with the same free variables), we write $\phi \leq_{\cal D} \psi$ if, for every $D\in {\cal D}$, we have $\phi(D) \leq \psi(D)$.  If $\phi$ is a pp formula, then we set $\langle \phi \rangle_{\cal D} = \{ \psi \mbox{ pp} \, : \phi \leq_{\cal D} \psi\}$ to be the {\bf pp-type generated by} $\phi$ {\bf in} or {\bf modulo} (the theory of) ${\cal D}$  and we will also say that this pp-type is ${\cal D}$-{\bf generated by} $\phi$.  A ${\cal D}${\bf -finitely generated} pp-type is one which is ${\cal D}$-generated by some pp formula.

More generally, if $q$ is any set of pp formulas which, for our purposes, we will assume to be closed under finite conjunction ($\wedge$ = ``and"), then we define the ${\cal D}$-{\bf closure} of $q$ to be $\langle q \rangle_{\cal D} = \{ \psi \mbox{ pp} \, : \psi \geq_{\cal D} \phi \mbox{ for some } \phi \in q\}$.

In these definitions, we drop the subscript ${\cal D}$ when ${\cal D} = {\rm Mod}\mbox{-}R$.
\end{definition}

That is, a pp-type $p$ is ${\cal D}$-generated by $\phi$ if $\phi \in p$ and if, in every $D\in {\cal D}$, every tuple which satisfies $\phi$ also satisfies every formula in $p$.  Thus $p$ is the smallest pp-type which contains $\phi$ and which is the pp-type of some tuple $\overline{a}$ of elements from some $D \in {\cal D}$.  Similarly, if a tuple $\overline{a}$ of elements from some $D\in {\cal D}$ satisfies all the formulas in $q$, then it will satisfy all those in $\langle q \rangle_{\cal D}$.

By \cite[3.5]{RadSaor} every definable subcategory ${\cal D}$ of a module category ${\rm Mod}\mbox{-}R$ is {\bf preenveloping} (as well as covering, e.g.~\cite[2.4]{CrivPreTorr}), that is, given any $M\in {\rm Mod}\mbox{-}R$ there is a morphism $f:M\to D_M\in {\cal D}$ - a ${\cal D}$-{\bf preenvelope} of $M$ - such that, for every morphism $g:M \to D\in {\cal D}$, there is $h:D_M \to D$ with $hf=g$.

$\xymatrix{M \ar[r]^f \ar[dr]_g & D_M \ar@{.>}[d]^h \\ & D}$

\noindent However, $h$ may well not be unique, and the choice of $D_M$ is not unique in any sense though, see \cite[3.3]{LackTend}, choosing such a weak reflection into ${\cal D}$ can be made functorial.

Recall also that a ${\cal D}$-{\bf envelope} of a module $M$ is a ${\cal D}$-preenvelope $f:M \to D_M$ with the property that any endomorphism $h:D_M \to D_M$ such that $hf=f$ is an automorphism.

\begin{lemma}\label{Denvpp} \marginpar{Denvpp} Let $A\in {\rm mod}\mbox{-}R$ and let $\overline{a}$ be a tuple from $A$ with ${\rm pp}^A(\overline{a}) = \langle \phi \rangle$.  If $f: A \to D \in {\cal D}$ is a ${\cal D}$-preenvelope of $A$, then ${\rm pp}^D(f\overline{a}) = \langle \phi \rangle_{\cal D}$.
\end{lemma}
\begin{proof}  Since morphisms preserve pp formulas, certainly $\phi \in {\rm pp}^D(f\overline{a})$.  If $D'\in {\cal D}$ and $\overline{b}\in \phi(D')$, then there is, by \cite[1.2.17]{PreNBK}, a morphism $A \to D'$ taking $\overline{a}$ to $\overline{b}$ and hence a morphism from $D$ to $D'$ taking $f\overline{a}$ to $\overline{b}$.  Therefore, ${\rm pp}^D(f\overline{a}) \subseteq {\rm pp}^{D'}(\overline{b})$.  That is, ${\rm pp}^D(f\overline{a})$ is the minimal pp-type realised in ${\cal D}$ containing $\phi$, as claimed.
\end{proof}

There is $\phi$ as in the statement of \ref{Denvpp} because every pp-type realised in a finitely presented module $A$ is finitely generated (by \cite[1.2.6]{PreNBK}).  We will use the fact, \cite[\S 1.2.2]{PreNBK}, that every pp formula $\phi$ has a {\bf free realisation} in ${\rm Mod}\mbox{-}R$, meaning a pair $(A,\overline{a})$ with $A$ finitely presented and ${\rm pp}^A(\overline{a}) = \langle \phi \rangle$.  So \ref{Denvpp} gives a weak relative version of this.  But it turns out that there is a stronger existence result:  see \ref{Mak}(b) below.

\begin{definition} Say that $M\in {\cal D}$ is ${\cal D}$-{\bf atomic}, if every pp-type realised in $M$ is ${\cal D}$-finitely generated.\footnote{The term ``pp-atomic" would be more accurate because ``atomic" means in model theory that every realised {\em complete} type is finitely generated and here we mean that every realised {\em pp-type} is finitely generated.  But, in this additive context, it is the pp-types which have a central role.}
\end{definition}

This definition, in a form that allows for $M\notin {\cal D}$, is made, and a number of properties developed, in \cite{RotHab}, though there the term ${\cal D}^{\rm d}$-Mittag-Leffler is mostly used rather than ${\cal D}$-atomic.

\begin{lemma}\label{Datds} \marginpar{Datds} (\cite[2.4]{RotHab}) Let ${\cal D}$ be a definable category.
Then the class of ${\cal D}$-atomic modules is closed under pure submodules and arbitrary direct sums.
\end{lemma}
\begin{proof} Closure under pure submodules is immediate from the definition since pp-types are preserved under pure embeddings.

For closure under direct sums, since the definition is a condition on finite tuples of elements, and pp-types are, as already remarked, unchanged in pure submodules (in particular, in direct summands), it is enough to prove the case where $I$ is finite, indeed, the case where $I = \{ 1, 2\}$.  Take $\overline{a} = (\overline{a}_1, \overline{a}_2) \in D_1\oplus D_2$ with $D_1, D_2 \in {\cal D}$.  Let $\phi_1$ be a pp formula which ${\cal D}$-generates the pp-type of $\overline{a}_1$ in $D_1$ (equally of $(\overline{a}_1, \overline{0})$ in $D_1 \oplus D_2$), and similarly take $\phi_2$ for $\overline{a}_2$.  Then by (the proof of) \cite[1.2.27]{PreNBK}, the pp-type of $\overline{a}$ in $D_1 \oplus D_2$ is ${\cal D}$-generated by the pp formula $\phi_1 + \phi_2$ (the sum of pp formulas in defined in Section \ref{secmodth}), as required.
\end{proof}

We say that $M \in {\cal D}$ is ${\cal D}$-{\bf pure-projective} if every pure epimorphism $D\to M$ with $D \in {\cal D}$ splits.  We will see below (\ref{Dpureepi}) that this is equivalent to the property that morphisms from $M$ lift over pure epimorphisms in ${\cal D}$.  First we recall a characterisation of {\bf pure epimorphisms} =  cokernels of pure monomorphisms.

\begin{prop}\label{pureepichar} \marginpar{pureepichar} (see \cite[2.1.14]{PreNBK}) A morphism $f:N\to M$ of $R$-modules is a pure epimorphism iff, for every tuple $\overline{a}$ from $M$ and every pp formula $\phi$ such that $\overline{a} \in \phi(M)$, there is $\overline{b} \in \phi(N)$ with $f\overline{b} =\overline{a}$.
\end{prop}

For ease of reference, we note the closure of definable subcategories under pullbacks of pure epimorphisms and pushouts of pure monomorphisms.

\begin{lemma}\label{pureepipbk} \marginpar{pureepipbk} If ${\cal D}$ is a definable subcategory of ${\rm Mod}\mbox{-}R$ and if the following diagram is a pullback with $M, D, D'$ all in ${\cal D}$ and with $p$ a pure epimorphism, then $X\in {\cal D}$.

$\xymatrix{X \ar[d] \ar[r] & M \ar[d]^f \\
D \ar[r]_p & D'}$

\end{lemma}
\begin{proof}  Recall that $X=\{(m,d): fm=pd\} \leq M \oplus D$ and the morphisms from $X$ are the restrictions of the projection maps.  We show that $X$ is pure in $M \oplus D$, from which the conclusion follows.

Suppose\footnote{To check purity it is enough to consider pp formulas in one free variable, see \cite[2.1.6]{PreNBK}; or just put a bar over $m$, $d$, etc.} that $M \oplus D \models \phi((m,d))$ with $fm=pd$; say $M\models \theta((m,d), (\overline{n}, \overline{e}))$ where $\phi$ is $\exists \overline{y} \, \theta(x, \overline{y})$ with $\theta$ quantifier-free, $\overline{n}$ from $M$ and $\overline{e}$ from $D$.  So $M\models \theta(m, \overline{n})$, hence $D'\models \theta(fm, f\overline{n})$; also $D\models \theta(d, \overline{e})$, hence $D' \models \theta(pd, p\overline{e})$.  Since $fm=pd$, this gives $D'\models \theta(0, f\overline{n} - p\overline{e})$.  Since $p$ is a pure epimorphism, there exists $\overline{b}$ from $D$ with $D\models \theta(0, \overline{b})$ and $p\overline{b}= f\overline{n} - p\overline{e}$.  That is, $p(\overline{b} + \overline{e}) = f\overline{n}$, and so $(\overline{n}, \overline{b} + \overline{e})$ is in $X$.

Since $D\models \theta(d, \overline{e})$ and $D\models \theta(0, \overline{b})$ we deduce $D\models \theta(d, \overline{b} + \overline{e})$.  Together with $M\models \theta(m, \overline{n})$, this gives $M \oplus D \models \theta((m,d), (\overline{n}, \overline{b} + \overline{e}))$.  Therefore $X \models \theta((m,d), (\overline{n}, \overline{b} + \overline{e}))$, hence $X \models \phi((m,d))$, as required.
\end{proof}

\begin{lemma}\label{puremonopo} \marginpar{puremonopo} If ${\cal D}$ is a definable subcategory of ${\rm Mod}\mbox{-}R$ and if the following diagram is a pushout with $M, D, D'$ all in ${\cal D}$ and with $i$ a pure monomorphism, then $Y\in {\cal D}$.

$\xymatrix{D' \ar[r]^i \ar[d]_f & D \ar[d] \\
M \ar[r] & Y}$

\end{lemma}
\begin{proof}  The pushout is the factor of $M\oplus D$ by the anti-diagonal image, $(f,-i)D'$, of $D'$ but this is a pure submodule of $M\oplus D$ because, if $M\oplus D \models \phi((fd,-id))$ for some pp formula $\phi$ and $d\in D'$, then $D\models \phi(id)$, so $D'\models \phi(d)$ since $D'$ is pure in $D$.  But then its image $(fd, -id))$ under $(f,-i)$ also satisfies $\phi$, as required.  Since definable categories are closed under pure-epimorphic images, $Y= M\oplus D /(f, -i)D' \in {\cal D}$.
\end{proof}

\begin{lemma}\label{Dpureepi} \marginpar{Dpureepi} Let ${\cal D}$ be a definable category.  Then $M\in {\cal D}$ is ${\cal D}$-pure-projective iff, given $D, D' \in {\cal D}$, $p:D \to D'$ a pure epimorphism and $f:M\to D'$, then there is a morphism $g:M \to D$ with $pg=f$.

$\xymatrix{ & M \ar[d]^f \ar@{.>}[dl]_g\\
D \ar[r]_p & D'}$
\end{lemma}
\begin{proof}  ($\Rightarrow$) Form the pullback as in \ref{pureepipbk}, noting that $X\to M$ is a pure epimorphism (e.g.~\cite[2.1.22]{PreNBK}) and use that $X$ as there is in ${\cal D}$ to deduce that $X\to M$ splits giving, composed with $X \to D$, the required morphism $g$.

The other direction follows by applying the property with $f=1_M$.
\end{proof}

\begin{prop} (\cite[3.9, 3.12]{RotHab}) Let ${\cal D}$ be a definable category.

\noindent (1) Every ${\cal D}$-pure-projective module is ${\cal D}$-atomic.

\noindent (2) Every countably generated ${\cal D}$-atomic module is ${\cal D}$-pure-projective.

\noindent (3) A module $M$ in ${\cal D}$ is ${\cal D}$-atomic iff every finite subset of $M$ is contained in a ${\cal D}$-pure-projective pure submodule of $M$, and this implies that every countable subset of $M$ is contained in a ${\cal D}$-pure-projective pure submodule of $M$.
\end{prop}
\begin{proof} The proofs are as in the non-relative case but we include proofs of (1) and (2) since they illustrate some of the techniques we use in the paper.

(1) Suppose that $M$ is ${\cal D}$-pure-projective.  Every module is a pure epimorphic image of a direct sum of finitely presented modules (e.g.~\cite[2.1.25]{PreNBK}):  say $p: \bigoplus_i \, A_i \to M$ is a pure epimorphism with each $A_i$ finitely presented; set $p_i:A_i \to M$ to be the $i$th component of $p$.  For each $i$ choose a ${\cal D}$-preenvelope $g_i:A \to D_i \in {\cal D}$ and a factorisation $f_ig_i =p_i$ of $p_i$.  Set $g=(g_i)_i$.  The morphism $f=(f_i):D= \bigoplus_i \, D_i \to M$ is, applying \ref{pureepichar}, a pure epimorphism, hence splits; let $h:M \to D$ be a splitting of $f$.  

Given $\overline{c}$ from $M$, choose $\overline{a}$ from $\bigoplus_i \, A_i$ with $g\overline{a} = h\overline{c}$, so $fg\overline{a} = \overline{c}$.  By \ref{Denvpp}, ${\rm pp}^D(f\overline{a})$ is ${\cal D}$-generated by any pp formula $\phi$ which generates ${\rm pp}^A(\overline{a})$ and, since $M$ is a direct summand of $D$, ${\rm pp}^M(\overline{c}) = {\rm pp}^D(h\overline{c})$ is ${\cal D}$-finitely generated by $\phi$, as required.

\noindent (2)  Suppose that $a_1, a_2, \dots, a_n, \dots$ is an enumeration of a countable set of generators for the ${\cal D}$-atomic module $M$ and suppose that $\pi:D \to M$ is a pure epimorphism.  Let $\phi_1 = \phi(x_1)$ ${\cal D}$-generate ${\rm pp}^M(a_1)$.  By assumption and \ref{pureepichar} there is $c_1\in \phi_1(D)$ with $\pi c_1 =a_1$.  Note that, since morphisms are non-decreasing on pp-types, ${\rm pp}^D(c_1)$ is therefore ${\cal D}$-generated by $\phi_1$.

Choose $\phi_2 = \phi_2(x_1, x_2)$ which ${\cal D}$-generates ${\rm pp}^M(a_1, a_2)$.  Then $\exists x_2 \, \phi_2(x_1, x_2) \in {\rm pp}^M(x_1) = {\rm pp}^D(c_1)$, so there is $c'_2\in D$ with $(c_1,c'_2) \in \phi_2(D)$.  Now, ${\rm pp}^D(c_1,c'_2)$ might strictly contain $\langle \phi_2 \rangle_{\cal D}$ but, since $\pi$ is a pure epimorphism, there is, as above, some $(b_1,b_2) \in \phi_2(D)$ with $\pi b_1 =a_1$ and $\pi b_2 =a_2$.  Then we have $D\models \phi_2(c_1-b_1, c'_2-b_2)$ and also $c_1 -b_1 \in {\rm ker}(\pi)$.  Since ${\rm ker}(\pi)$ is pure in $D$, there is $d_2\in {\rm ker}(\pi)$ with $D\models \phi_2(c_1-b_1, d_2)$.  Combining with $\phi_2(b_1, b_2)$ gives $D\models \phi_2(c_1, b_2+d_2)$; set $c_2=b_2+d_2$.  Noting that $\pi : (c_1, c_2) \mapsto (a_1, a_2)$, we conclude that ${\rm pp}^D(c_1, c_2) = \langle \phi_2 \rangle _{\cal D} = {\rm pp}^M(a_1, a_2)$.

We continue in this way, to obtain $c_1, c_2, \dots \in D$ with the same pp-type as $a_1, a_2, \dots$.  In particular we have a map $f:M \to D$, well-defined by $fa_i =c_i$ (since any $R$-linear relation between $a_1,\dots, a_n$ is part of ${\rm pp}^M(a_1,\dots, a_n)= {\rm pp}^C(c_1,\dots, c_n)$), splitting $\pi$, as required.
\end{proof}

\begin{rmk} It is shown in \cite[3.1]{KucRot} that a pp-constructible module is pure-projective, where a module $M$ is {\bf pp-constructible} if it is the union $M=\bigcup_{i<\alpha} \, A_i$ of subsets where, for each $i \geq -1$, $A_{i+1}$ is the union of $A_i$ (take $A_{-1} = \emptyset$) and (the entries of) some finite tuple $\overline{a}_i$ of elements of $M$ such that the pp-type, ${\rm pp}^M(\overline{a}_i/A_i)$, of $\overline{a}_i$ in $M$ over $A_i$ is finitely generated.  Again, this - both definition and result - can be relativised to a definable category ${\cal D}$ by taking $M\in {\cal D}$ and requiring the pp-types ${\rm pp}^M(\overline{a}_i/A_i)$ to be ${\cal D}$-finitely generated.
\end{rmk}

To continue with some degree of self-containedness, we now give a proof (essentially that of Rothmaler \cite[2.2]{RotHab}) of the relative version of \ref{MLchar} (at least, of the equivalence of (ii) and (iii) there).

First we recall Herzog's criterion for a tensor to be $0$.  Here $D$ denotes elementary duality, see Section \ref{secmodth} (and recall from there that $M\models \phi(\overline{a})$ means the same as $\overline{a} \in \phi(M)$).

\begin{theorem}\label{Herz} \marginpar{Herz} (\cite[3.2]{HerzDual}) If $\overline{a}$ is a tuple of elements from a right $R$-module $M$ and $\overline{l}$ is a tuple of the same length from a left $R$-module $L$, then $\overline{a} \otimes \overline{l} =0$ in $M\otimes_RL$ (that is, $\sum_{i=1}^n \, a_i \otimes_R l_i =0$) iff there is a pp formula $\phi(\overline{x})$ for right $R$-modules such that $M\models \phi(\overline{a})$ and $L \models D\phi(\overline{l})$.
\end{theorem}

Also (again, see Section \ref{secmodth}) if ${\cal D}$ is a definable subcategory of ${\rm Mod}\mbox{-}R$, then ${\cal D}^{\rm d}$ denotes its elementary dual definable category - a definable subcategory of $R\mbox{-}{\rm Mod}$.

\begin{theorem} \label{relMLchar} \marginpar{relMLchar} Suppose that ${\cal D}$ is a definable subcategory of ${\rm Mod}\mbox{-}R$ and that $M\in {\rm Mod}\mbox{-}R$.  Then the following conditions are equivalent:

\noindent (i) for all sets $L_i \in {\cal D}^{\rm d}$ ($i\in I$) the canonical morphism $t: M \otimes_R \prod_i \, L_i \to \prod_i \, M\otimes_R L_i$ is monic;

\noindent (ii) every pp-type realised in $M$ is ${\cal D}$-finitely generated; that is, for every finite tuple $\overline{a}$ from $M$, there is a pp formula $\phi \in {\rm pp}^M(\overline{a})$ such that, for every $\psi \in {\rm pp}^M(\overline{a})$, we have $\phi \leq_{\cal D} \psi$.
\end{theorem}
\begin{proof} (ii)$\Rightarrow$(i) Given a set $(L_i)_i$ of modules in ${\cal D}^{\rm d}$, suppose that we have a tuple $\overline{q} = (\overline{q}_i)_i \in \prod_i \, L_i$ and matching tuple $\overline{a}$ from $M$ such that $t(\overline{a}\otimes \overline{q}) =0$.  That is $\overline{a}_i \otimes \overline{q}_i =0$ for all $i$.  Then, by Herzog's criterion \ref{Herz} there, for each $i$, is a pp formula $\psi_i$ such that $M\models \psi_i(\overline{a})$ and $L_i \models D\psi_i(\overline{q}_i)$.

Let $\phi$ be a pp formula which generates, with respect to $\leq_{\cal D}$, the pp-type of $\overline{a}$ in $M$.  Then, for each $i$, $\phi \leq_{\cal D} \psi_i$ so (Section \ref{secmodth}) $D\psi_i \leq_{{\cal D}^{\rm d}} D\phi$ and hence $L_i \models D\phi(\overline{q}_i)$.  That is true for every $i$, so $\prod_i \, L_i \models D\phi(\overline{q})$.  So, again by Herzog's Criterion and since $M\models \phi(\overline{a})$, we have $\overline{a} \otimes \overline{q} =0$ in $M\otimes \prod\, L_i$, and so $t$ is monic as claimed.

(i)$\Rightarrow$(ii)  The proof above essentially reverses.  Given $\overline{a}$ from $M$, for each $\psi_i \in {\rm pp}^M(\overline{a})$ choose $L_i \in {\cal D}^{\rm d}$ to contain a tuple $\overline{q}_i$ such that ${\rm pp}^{L_i}(\overline{q}_i)$ is ${\cal D}^{\rm d}$-generated by $D\psi_i$ - for instance, noting \ref{Denvpp}, take $L_i$ to be a ${\cal D}^{\rm d}$-preenvelope of a free realisation of $D\psi_i$ in $R\mbox{-}{\rm Mod}$.  Since $M\models \psi_i(\overline{a})$ and $L_i \models D\psi_i(\overline{q}_i)$, we have $\overline{a} \otimes \overline{q_i} =0$ in $M\otimes L_i$.

By assumption, it follows that $\overline{a} \otimes \overline{q} =0$ in $M \otimes \prod_i \, L_i$ where $\overline{q} = (\overline{q}_i)_i$.  So there is a pp formula $\phi$ such that $M\models \phi(\overline{a})$ and $L_i \models D\phi(\overline{q}_i)$ for all $i$.  By choice of $L_i$ and $\overline{q}_i$, $D\psi_i \leq_{{\cal D}^{\rm d}} D\phi$, hence $\phi \leq_{\cal D} \psi_i$ for every $i$.  That is, $\phi \leq_{\cal D} \psi$ for every $\psi \in {\rm pp}^M(\overline{a})$, as required.
\end{proof}

\section{Strictly atomic modules}\label{secSML} \marginpar{secSML}

A right module $M$ is said to be {\bf strictly Mittag-Leffler} if, for every tuple $\overline{a}$ from $M$, there is a finitely presented module $A$ and a pair, $f:M\to A$, $g:A\to M$, of morphisms such that $gf\overline{a} =\overline{a}$.  Since pp-types realised in finitely presented modules are finitely generated and since morphisms are non-decreasing on pp-types, it follows that there is a pp formula $\phi$ such that the pp-type, ${\rm pp}^M(\overline{a})$ of $\overline{a}$ in $M$ is generated by $\phi$ (and $(A, f\overline{a})$ is a free realisation of $\phi$).  Thus we obtain the following characterisation.

\begin{lemma} \label{sMLchar} \marginpar{sMLchar} A module $M$ is strictly Mittag-Leffler iff $M$ is Mittag-leffler and if, for every tuple $\overline{a}$ from $M$ and pp formula $\phi$ such that ${\rm pp}^M(\overline{a}) = \langle \phi \rangle$, if $N$ is any module and $\overline{b} \in \phi(N)$, then there is a morphism $f:M\to N$ with $f\overline{a} = \overline{b}$.
\end{lemma}
\begin{proof}  For ($\Rightarrow$), take $A$ in the definition to be a free realisation of $\phi$, noting that there will then be a morphism from $M$ to $A$ and another from $A$ to $N$.  For the other direction, again take $A$ to be a free realisation of $\phi$.  
\end{proof}

Definable categories do not, in general, have enough finitely presented objects (those which do are exactly the locally finitely presented categories with products), indeed they may have $0$ as the only finitely presented object \cite[18.1.1]{PreNBK} but the property above - which, \cite[1.2.7]{PreNBK}, is a property of finitely presented modules - does generalise.  Indeed, it will be the strictly ${\cal D}$-atomic modules, defined below and generalising the strictly Mittag-Leffler modules, that are the next best thing to finitely presented objects in definable categories.  Makkai \cite[4.1]{MakkTop}, \cite[4.4]{Makk} proved that there is a $\varinjlim$-generating set of these in every definable category.  In fact, his result in \cite{Makk} is more general in two directions:  it includes infinitary versions (which allow infinitary pp formulas and infinite tuples of elements) and his results apply in general categories of models of regular theories (see \cite{Butz} for these).  We will come back to his result but now we consider the following property equivalent to being strictly ML.

An epimorphism $f:N\to M$ is {\bf locally split} if, for every tuple $\overline{a}$ from $M$ there is a `local section', that is a morphism $g:M \to N$ such that $fg\overline{a} = \overline{a}$.  A module is {\bf locally pure-projective} \cite{Azu} if every pure epimorphism to it locally splits.

\begin{prop} (\cite[Thm.~5]{Azu}) A module is strictly Mittag-Leffler iff it is locally pure-projective.
\end{prop}

\begin{definition} Given a definable category ${\cal D}$, we say that a module $M\in {\cal D}$ is {\bf strictly} ${\cal D}$-{\bf atomic} if it is ${\cal D}$-atomic and if, for every tuple $\overline{a}$ from $M$, with pp-type ${\cal D}$-generated by, say, $\phi$, and for every $D\in {\cal D}$ and $\overline{b}\in \phi(D)$, there is a morphism $f:M\to D$ with $f\overline{a} = \overline{b}$.

Say that $M \in {\cal D}$ is {\bf locally} ${\cal D}${\bf -{pure-projective}} if every pure epimorphism $D \to M$ with $D\in {\cal D}$ locally splits.
\end{definition}

Note that \cite[2.1]{RotSML} allows a more general relative notion in that the module $M$ is not required to be in ${\cal D}$.

We will show (\ref{DateqDpp} below) that the strictly ${\cal D}$-atomic objects are exactly the locally ${\cal D}$-pure-projectives. 

\vspace{4pt}

In \cite{Makk} Makkai uses the term {\bf principal prime} or {\bf pp} object of ${\cal D}$ for what we have termed strictly ${\cal D}$-atomic.

\begin{lemma}\label{dsstrDat} \marginpar{dsstrDat} (\cite[2.5]{RotSML}) Suppose that ${\cal D}$ is a definable category.  Every pure submodule of a strictly ${\cal D}$-atomic module is strictly ${\cal D}$-atomic and every direct sum of strictly ${\cal D}$-atomic modules is strictly ${\cal D}$-atomic.
\end{lemma}
\begin{proof}  Since $N$ is pure in $M$, given any tuple $\overline{a}$ from $N$, we have ${\rm pp}^N(\overline{a})={\rm pp}^M(\overline{a})$, so ${\rm pp}^N(\overline{a}) = \langle \phi\rangle_{\cal D}$ for some pp $\phi$.  Now, if $\overline{b} \in \phi(D)$ for some $D\in {\cal D}$ then, since $M$ is strictly ${\cal D}$-atomic, there is a morphism $f:M \to D$ with $f\overline{a} = \overline{b}$; then we restrict $f$ to $N$.

For the second statement, as in \ref{Datds} it is enough to show that the direct sum of two strictly ${\cal D}$-atomic modules is strictly ${\cal D}$-atomic.

So suppose that $D_1, D_2$ are strictly ${\cal D}$-atomic.  As in the proof of \ref{Datds}, let $\overline{b} = (\overline{b}_1, \overline{b}_2)$ be a tuple in $D_1\oplus D_2$, with $\overline{b}_i \in D_i$ and $\phi_i$ a ${\cal D}$-generator for the pp-type of $\overline{b}_i$ in $D_i$, equally in $D_1 \oplus D_2$.  Then $\phi_1 + \phi_2$ is a ${\cal D}$-generator for the pp-type of $\overline{b}$ in $D_1 \oplus D_2$ (proof of \cite[1.2.17]{PreNBK}).   Suppose that $D\in {\cal D}$, and that $\overline{b} \in (\phi_1 + \phi_2)(D)$.  Then there are $\overline{b}_1 \in \phi_1(D)$ and $\overline{b}_2 \in \phi_2(D)$ with $\overline{b}_1 + \overline{b}_2 = \overline{b}$.  Since $D_i$ is strictly ${\cal D}$-atomic, there are $f_i:D_i \to D$ with $f_i\overline{a}_i = \overline{b}_i$, $i=1,2$; these combine to give $f=(f_1, f_2): D_1 \oplus D_2 \to D$ with $f\overline{a} = \overline{b}$, as required.
\end{proof}

We show that, for any definable subcategory ${\cal D}$, the strictly ${\cal D}$-atomic modules coincide with the strict ${\cal D}$-stationary modules in ${\cal D}$, defined in \cite[\S 8, esp.~8.11]{AngHerML}.  Let $D\in {\rm Mod}\mbox{-}R$; a module $M$ is said to be {\bf strict} $D$-{\bf stationary} if for every tuple $\overline{a}$ from $M$, there is $C\in {\rm mod}\mbox{-}R$, a matching tuple $\overline{c}$ from $C$ such that 

\noindent (i) there is a morphism $h:C \to M$ with $h\overline{c} = \overline{a}$ and 

\noindent (ii) for every matching tuple $\overline{d}$ from $D$, there is a morphism $f:C \to D$ with $f\overline{c} =\overline{d}$ iff there is a morphism $g:M \to D$ with $g\overline{a} = \overline(d)$.

\noindent More generally, if ${\cal D}$ is a subclass of ${\rm Mod}\mbox{-}R$, say that $M$ is {\bf strict} ${\cal D}$-{\bf stationary} if $M$ is strict $D$-stationary for every $D\in {\cal D}$.

\begin{prop}\label{streqstr} Suppose that ${\cal D}$ is a definable subcategory of ${\rm Mod}\mbox{-}R$ and let $M\in {\cal D}$.  Then $M$ is strictly ${\cal D}$-atomic if and only if $M$ is strict ${\cal D}$-stationary.
\end{prop}
\begin{proof}. ($\Rightarrow$) Take $\overline{a}$ from $M$; by assumption the pp-type of $\overline{a}$ in $M$ is ${\cal D}$-generated by some pp formula $\phi$.  Let $(C, \overline{c})$ be a free realisation of $\phi$.  Then there is a morphism $h:C \to M$ taking $\overline{c}$ to $\overline{a}$.  Any morphism $(M,\overline{a})$ to $(D,\overline{d})$ with $D\in {\cal D}$ yields, by composition with $h$, a morphism $(C,\overline{c})$ to $(D,\overline{d})$.  And, in the other direction, given $f:C \to D\in {\cal D}$, the image $\overline{d} = f\overline{c}$ is in $\phi(D)$ so, by assumption on $M$, there is $g:M\to D$ taking $\overline{a}$ to $\overline{d}$, as required.

($\Leftarrow$) Suppose that $M$ is strict ${\cal D}$-stationary and let $\overline{a}$ be a tuple from $M$.  Let $(C, \overline{c})$ be as in the definition of strict stationarity and let $\phi$ be a generator of the pp-type of $\overline{c}$ in $C$.  Since there is a morphism from $C$ to $M$ taking $\overline{c}$ to $\overline{a}$, certainly $\phi \in {\rm pp}^M(\overline{a})$.  We show that $\phi$ ${\cal D}$-generates ${\rm pp}^M(\overline{a})$.  That is, we show that if $\overline{d} \in \phi(D) \in {\cal D}$, then ${\rm pp}^D(\overline{d}) \supseteq {\rm pp}^M(\overline{a})$.  If $\overline{d} \in \phi(D)$ then, since $(C, \overline{c})$ freely realises $\phi$, there is a morphism $f:C \to D$ taking $\overline{c}$ to $\overline{d}$.  So, by assumption, there is $g:M\to D$ with $g\overline{a} =\overline{d}$ and so, indeed, ${\rm pp}^D(\overline{d}) \supseteq {\rm pp}^M(\overline{a})$.  That argument shows that $M$ is ${\cal D}$-atomic and, at the same time, that $M$ is strictly ${\cal D}$-atomic.
\end{proof}

We note next that there is an internal-to-${\cal D}$ category-theoretic characterisation of strictly ${\cal D}$-atomic modules.  Reduced products are directed colimits of products (e.g.~see \cite[\S 3.3.1]{PreNBK}) so, since definable categories have products and directed colimits they have reduced products.

\begin{theorem} \label{charsML} \marginpar{charsML}  Let ${\cal D}$ be a definable category.  A module $M\in {\cal D}$ is strictly ${\cal D}$-atomic iff there is an index set $\Lambda$ and filter ${\cal F}$ on $\Lambda$ such that, whenever $\pi:P\to M$ is a pure epimorphism with $P\in {\cal D}$, there are morphisms $f_\lambda:M \to P$ ($\lambda \in \Lambda$) such that, if $\pi^\ast = \pi^\Lambda/{\cal F} :P^\ast = P^\Lambda/{\cal F} \to M^\ast = M^\Lambda/{\cal F}$ denotes the corresponding reduced product, the morphism $f:M^\ast \to P^\ast$ which is $(f_\lambda)_\lambda/{\cal F}$ satisfies $\pi^\ast f \Delta_M = \Delta_M$, where $\Delta_M:M\to M^\ast$ is the diagonal embedding.
\end{theorem}
\begin{proof} Suppose that $M$ is strictly ${\cal D}$-atomic.

Let $\Lambda$ be the set of finite subsets, which we write as tuples, of $M$.  Consider the filter-base consisting of the sets of the form $\langle \overline{a} \rangle = \{ \overline{b}: \overline{a} \subseteq \overline{b}\}$ and let ${\cal F}$ be any filter containing this filter-base.  Denote by $\Delta_P:P\to P^\ast$ and $\Delta_M:M \to M^\ast$ the canonical (pure) embeddings into the corresponding reduced products.

Now, given any pure epimorphism $\pi:P\to M$ with $P \in {\cal D}$, choose, for each $\overline{a} \in \Lambda$, some local splitting $f_{\overline{a}}:M\to P$ such that $\pi f_{\overline{a}} (\overline{a}) =\overline{a}$.

Define $f:M^\ast \to P^\ast$ by $c^\ast= (c_{\overline{a}})_{\overline{a}}/{\cal F} \mapsto (d_{\overline{a}})_{\overline{a}}/{\cal F}$ where $d_{\overline{a}} = f_{\overline{a}}(c_{\overline{a}})$ if $c_{\overline{a}} \in \langle \overline{a} \rangle$ and $d_{\overline{a}} =0$ otherwise.  

Note that $f$ is well-defined since, if $(c_{\overline{a}})_{\overline{a}}/{\cal F} = (b_{\overline{a}})_{\overline{a}}/{\cal F}$, then $\{ \overline{a}: c_{\overline{a}}=b_{\overline{a}}\} \in {\cal F}$ and hence $\{ \overline{a}: fc_{\overline{a}}=fb_{\overline{a}}\} \supseteq \{ \overline{a}: c_{\overline{a}}=b_{\overline{a}}\}$ also is in ${\cal F}$, as required.

We show that $\pi^\ast f \Delta_M = \Delta_M$.  So take $c\in M$.  Then $\pi^\ast f \Delta_M(c) = (\pi f_{\overline{a}} (c))_{\overline{a}}/{\cal F}$.  By construction of ${\cal F}$, we have  $\langle c \rangle \in {\cal F}$ and, if $\overline{a} \in \langle c\rangle$, that is, if $c\in \overline{a}$, then $\pi f_{\overline{a}}(c) =c$.  Therefore $\pi^\ast f \Delta_M = \Delta_M$, as required.

For the converse, if $M$ satisfies this condition, then let $\pi:P \to M$ be a pure epimorphism and let $f_\lambda:M\to P$ ($\lambda \in \Lambda$) be morphisms as described.  Let $\overline{a}$ be a finite subset of $M$.  By assumption the morphism $f= (f_\lambda)_\lambda/{\cal F}:M^\ast \to P^\ast$ satisfies $(\pi^\ast f) (\overline{a})_\lambda/{\cal F} = (\overline{a})_\lambda/{\cal F}$.  In particular there is some (indeed there are many) $\lambda$ with $\pi f_\lambda \overline{a} = \overline{a}$, showing that $M$ is locally ${\cal D}$-pure-projective hence (\ref{DateqDpp} below) strictly ${\cal D}$-atomic.
\end{proof}

We denote by $(-)^\ast$ the hom-dual of a module taken with respect to an injective cogenerator for the category of modules over some chosen subring of its endomorphism ring; we will suppose where needed that the injective cogenerator is minimal or at least that each of its indecomposable direct summands is the injective hull of a simple module.  So $M^\ast$ could be ${\rm Hom}_k(M,k)$ if $k$ is a field and $R$ is a $k$-algebra, it could be ${\rm Hom}_{\mathbb Z}(M, {\mathbb Q}/{\mathbb Z})$, or ${\rm Hom}_S(M,E)$ where $S={\rm End}(M_R)$ and $E$ is a minimal injective cogenerator of $S\mbox{-}{\rm Mod}$.

\begin{lemma}\label{fgsub} \marginpar{fgsub} (\cite[2.12-14]{RotSML}) If ${\cal D}$ is a definable category, $M\in {\cal D}$ is strictly ${\cal D}$-atomic and $S={\rm End}(M)$, then every finitely generated $S$-submodule of $M$ is pp-definable.
\end{lemma}
\begin{proof}  If $a\in M$, take $\phi$ which ${\cal D}$-generates ${\rm pp}^M(a)$.  Then $Sa = \phi(M)$:  the containment $Sa \leq \phi(M)$ since morphisms preserve pp formulas and the converse because, if $b\in \phi(M)$ then, since $M$ is strictly ${\cal D}$-atomic, there is $f\in S$ with $fa=b$. More generally, if we have $a_1, \dots, a_n \in M$, and take $\phi_i$ to ${\cal D}$-generate ${\rm pp}^M(a_i)$, then $\sum_{i=1}^n \, Sa_i = \sum_{i=1}^n \, \phi_i(M) = \phi'(M)$, where $\phi'$ is the pp formula $\sum_{i=1}^n \, \phi_i$.
\end{proof}

A pp-type is said to be {\bf neg-isolated} by a pp formula $\phi$ if it is maximal among pp-types with respect to not containing $\phi$.  Any such pp-type is irreducible, so is realised in an indecomposable pure-injective.  More generally, if ${\cal D}$ is a definable subcategory, then a pp-type $p$ is said to be ${\cal D}$-{\bf neg-isolated} (or neg-isolated with respect to ${\cal D}$) if there is a pp formula $\phi$ such that $p$ is maximal among pp-types of $n$-tuples of elements in modules in ${\cal D}$ with respect to not containing $\phi$.  Again, any such pp-type is irreducible, hence realised in an indecomposable pure-injective in ${\cal D}$.  See Section \ref{secmodth}.

\begin{theorem}\label{dualnegisol} \marginpar{dualnegisol}  Suppose that ${\cal D}$ is a definable subcategory of ${\rm Mod}\mbox{-}R$ and $M \in {\cal D}$ is strictly ${\cal D}$-atomic.  Set $S={\rm End}(M)$ and let $M^\ast = {\rm Hom}_S(_SM,\, _SE)$ where $_SE$ is a minimal injective cogenerator of $S\mbox{-}{\rm Mod}$.  Then every indecomposable direct summand of $M^\ast$ is neg-isolated with respect to the definable subcategory generated by $M^\ast$.
\end{theorem}
\begin{proof}  Let $f\in M^\ast$ and set $p= {\rm pp}^{M^\ast}(f)$.  We use that $M^\ast \models \phi(f)$ iff $D\phi(M) \leq {\rm ker}(f)$ (\cite[1.5]{PRZ2}, \cite[\S 2(c)]{Z-HZPSS}, see \cite[1.3.12]{PreNBK}), that is, $\phi\in p$ iff $fD\phi(M)=0$.

Suppose that $p$ is irreducible so, see Section \ref{secmodth}, the hull of $f$ in $M^\ast$ is a typical indecomposable summand of $M^\ast$.  We show that $M/{\rm ker}(f)$ is a uniform $S$-module.  

Suppose that $a,b\in M \setminus {\rm ker}(f)$.  Since $M$ is ${\cal D}$-atomic, there is a pp formula $\psi_1$ such that ${\rm pp}^M(a) = \langle \psi_1 \rangle_{\cal D}$; since $a \notin {\rm ker}(f)$,  $D\psi_1 \notin p$ (recall that $D^2$ is the identity on pp formulas).  Similarly, there is a pp formula $\psi_2$ such that ${\rm pp}^M(b)$ is ${\cal D}$-generated by $\psi_2$ and so with $D\psi_2 \notin p$.  Since $p$ is irreducible there is, see \cite[\S 4.3.6]{PreNBK} (Ziegler's Criterion), a pp formula $\phi \in p$ such that $(\phi \wedge D\psi_1) + (\phi \wedge D\psi_2) \notin p$.  Therefore the solution set in $M$ of the dual pp formula $(D\phi + \psi_1) \wedge (D\phi + \psi_2)$ is not contained in ${\rm ker}(f)$.  Therefore, since $D\phi(M) \leq {\rm ker}(f)$, we have that $f\psi_1(M) \cap f\psi_2(M) \neq 0$.  That is, by \ref{fgsub}, $fSa \cap fSb \neq 0$  and so the images of $Sa$ and $Sb$ under $f$ have non-zero intersection, showing that $M/{\rm ker}(f)$ is indeed a uniform $S$-module.

Therefore $fM \simeq M/{\rm ker}(f)$ is contained in an indecomposable direct summand $E'$ of $E$.  By choice of $E$, $E'$ has a non-zero simple $S$-submodule, which necessarily lies in the image of $f$, so let $a\in M$ be such that $fa$ generates that simple module.  By assumption, there is a pp formula $\psi$ which ${\cal D}$-generates ${\rm pp}^M(a)$ and, by \ref{fgsub}, $Sa=\psi(M)$, so the simple submodule of $E'$ is $f\psi(M)$.  We claim that $p$ is neg-isolated, for the definable subcategory $\langle M^\ast\rangle$ generated by $M^\ast$, by $D\psi$.

To see that, suppose that $q$ is a pp-type for $\langle M^\ast\rangle$ (that is, such that $\phi \in q$ and $\phi \leq_{M^\ast} \psi$ implies $\psi \in q$), strictly containing $p$; say $\eta \in q \setminus p$.  Then $D\eta(M)$ is not contained in ${\rm ker}(f)$ and so $fa \in fD\eta(M)$ (since $fSa$ is the unique minimal $S$-submodule of $fM$).  Therefore $a=b+c$ for some $b\in D\eta(M)$ and $c\in {\rm ker}(f)$.  Choose $\phi$ to ${\cal D}$-generate ${\rm pp}^M(c)$; so $a \in D\eta(M) + \phi(M)$.  Since, by \ref{fgsub}, $Sa = \psi(M)$, we deduce that $\psi(M) \leq D\eta(M) + \phi(M)$.  Hence, by elementary duality, $\eta (M^\ast) \cap D\phi(M^\ast) \leq D\psi(M^\ast)$, that is $\eta \wedge D\phi \leq_{M^\ast} D\psi$.  Since $\phi(M) =Sc \leq {\rm ker}(f)$, we have $D\phi \in p \subseteq q$, hence $\eta \wedge D\phi \in q$ (pp-types are closed under $\wedge$).  Therefore, since $\eta \wedge D\phi \leq_{M^\ast} D\psi$, $D\psi \in q$, as required.
\end{proof}

Note the special case ${\cal D} = {\rm Mod}\mbox{-}R$ and so the module finitely presented.

\begin{cor}\label{dualnegisolfp} \marginpar{dualnegisolfp}  If $A$ is a finitely presented $R$-module, $S$ is its endomorphism ring and $E$ is a minimal injective cogenerator of $S\mbox{-}{\rm Mod}$, then every indecomposable direct summand of the dual $A^\ast = {\rm Hom}_S(A,E)$ is ${\cal D}$-neg-isolated where ${\cal D}$ is the definable category generated by $A^\ast$.
\end{cor}

We already know, by \cite[3.5]{MehPre}, that the dual module $M^\ast$ has `enough' neg-isolated, in particular indecomposable, direct summands.  The theorem \ref{dualnegisol} above says that, for a strictly ${\cal D}$-atomic module, and for the specific type of duality chosen above, {\em every} indecomposable direct summand is neg-isolated.  There can however, as the following example illustrates, be superdecomposable direct summands of the dual of a finitely presented module, where a module is said to be {\bf superdecomposable} if it is nonzero and has no indecomposable direct summand.

\begin{example}  Let $R$ be a simple, non-artinian, von Neumann regular ring; the regularity condition implies that every embedding between $R$-modules is pure and hence that the pure-injectives are exactly the injectives.  The module (left or right) $R$ has no uniform submodules (see \cite[7.3.19]{PreNBK}) so its injective hull is superdecomposable.

Consider the left module $_RR$; this is strictly atomic for the whole category $R\mbox{-}{\rm Mod}$ so, noting that ${\rm End}(_RR) =R$, consider the right module $(_RR)^\ast = {\rm Hom}(R_R, E(R_R))$.  The canonical embedding $f:R_R \to E(R_R)$ is an element of this dual module and it generates a copy of $R_R$.  Thus $R_R$ embeds, purely since $R$ is regular, in $(_RR)^\ast$, hence the superdecomposable (pure-)injective $E(R_R)$ is a direct summand of $(_RR)^\ast$, as required.
\end{example}

\subsection{Strictly atomic generators}\label{secsatgen} \marginpar{secsatgen}

Makkai \cite{Makk} proves a remarkably strong result, a special case of which we state now.  In fact, this statement reflects some of his proof, not simply his formally-stated conclusion(s).

\begin{theorem}\label{Mak} \marginpar{Mak} (\cite[4.4, 4.3]{Makk}, see also \cite[4.1]{MakkTop}) 

\noindent (a)  Let ${\cal D}$ be a definable category.  Then there is a $\varinjlim$-generating set of strictly ${\cal D}$-atomic modules in ${\cal D}$.

\noindent (b)  Suppose that ${\cal D}$ is a definable subcategory of ${\rm Mod}\mbox{-}R$ and let $A\in {\rm mod}\mbox{-}R$ be any finitely presented $R$-module.  Then there is a ${\cal D}$-preenvelope $A\to D_A$ where $D_A$ is strictly ${\cal D}$-atomic.
\end{theorem}

\begin{remarks}
In part (a) it is the existence of enough strictly ${\cal D}$-atomic models which is the point.  That one can take a set of them to $\varinjlim$-generate is direct from the Downwards L\"{o}wenheim-Skolem Theorem; or one can use those appearing in part (b).

\vspace{4pt} 

We noted earlier that, if $A$ is in ${\rm mod}\mbox{-}R$, if $\overline{a}$ is a tuple from $A$, and if $f:A\to D_A \in {\cal D}$ is {\em any} ${\cal D}$-preenvelope then the pp-type of $f\overline{a}$ in $D_A$ will be ${\cal D}$-finitely generated, indeed, \ref{Denvpp}, will be ${\cal D}$-generated by any pp formula which generates the pp-type of $\overline{a}$ in $A$.  If $\overline{a}$ generates the module $A$, we can take this pp formula to be quantifier-free (specifying finitely many relations which define the module $\sum_{i=1}^n\, a_iR$).  Clearly these $D_A$, as $A$ ranges over finitely presented $R$-modules, form a $\varinjlim$-generating subset of ${\cal D}$.  And $D_R$ is even a generator in the sense that every $D\in {\cal D}$ is an epimorphic image of a coproduct of copies of $D_R$.  But, though the pp-type of $\overline{a}$ in $D_A$ is finitely generated, there is no reason in general to suppose that {\em every} tuple in $D_A$ has finitely generated pp-type - i.e.~that $D_A$ is ${\cal D}$-atomic, let alone strictly ${\cal D}$-atomic.  Makkai shows that there is, nevertheless, some choice of $A\to D_A$ such that $D_A$ is strictly ${\cal D}$-atomic.
\end{remarks}

Makkai's construction/proof is a model-theoretic Henkin-style construction, originally appearing as \cite[4.1]{MakkTop} and done in great generality in \cite{Makk}.  It is perhaps not easy to extract its core from the surrounding details but, in the case that the ring $R$ is countable, we can give what we hope is a more conceptual and algebraic proof which makes the relation between the inputs and outputs of the construction clearer.  That proof, which was found in discussion with Philipp Rothmaler, is given in the next section.  Here we derive some consequences of the existence of ``enough" strictly atomic models.  We begin by noting that (b) of \ref{Mak} extends to pure-projective $R$-modules.

\begin{cor}  If ${\cal D}$ a definable subcategory of ${\rm Mod}\mbox{-}R$, then every pure-projective $R$-module has a strictly ${\cal D}$-atomic ${\cal D}$-preenvelope.
\end{cor}
\begin{proof} If, for $i\in I$, $A_i$ is finitely presented and $f_i:A_i \to D_i$ is a strictly ${\cal D}$-atomic ${\cal D}$-preenvelope, then $\bigoplus_i \, f_i$ is clearly (reduce to the finite case) a strictly ${\cal D}$-atomic ${\cal D}$-preenvelope for $\bigoplus_i \, A_i$, and hence for any direct summand of $\bigoplus_i \, A_i$, that is, for every pure-projective $R$-module.
\end{proof}

\begin{lemma} \label{pureepi} \marginpar{pureepi}  If ${\cal D}$ is a definable category and $D\in {\cal D}$ then there is a strictly ${\cal D}$-atomic $M\in {\cal D}$ and a pure epimorphism $M\to D$.
\end{lemma}
\begin{proof} There is a pure epimorphism $f:P\to D$ where $P=\bigoplus_i\, A_i$ is a direct sum of finitely presented $R$-modules, see \cite[2.1.25]{PreNBK}.  Each component map from some $A_i$ to $D$ factors through $A_i\to D_{A_i}$ where $D_{A_i}$ is a strictly ${\cal D}$-atomic ${\cal D}$-preenvelope of ${A_i}$.  Take $M$ to be the, strictly ${\cal D}$-atomic by \ref{dsstrDat}, direct sum of these modules $D_{A_i}$; it is direct to check that the corresponding map $M\to D$ is a pure epimorphism.
\end{proof}

In particular, every ${\cal D}$-pure-projective in ${\cal D}$ is a direct summand of a direct sum of strictly ${\cal D}$-atomic ${\cal D}$-preenvelopes of finitely presented modules.

Since every definable category ${\cal D}$ is closed under pure subobjects and since a pure subobject of a strictly atomic object is strictly atomic (\ref{dsstrDat}), we deduce that every object of ${\cal D}$ has a pure presentation by strictly ${\cal D}$-atomic objects.

\begin{cor}  If ${\cal D}$ is a definable category and $D\in {\cal D}$ then there is a pure-exact sequence  $0 \to M_1 \to M_0 \to D \to 0$ with $M_0, M_1$ strictly ${\cal D}$-atomic.
\end{cor}

\begin{rmk}  It follows from \ref{pureepi} and \cite[3.7]{RotSML} that, in the definition of strictly ${\cal D}$-atomic, it is enough to require the ``free realisation" property for single elements (it then follows for finite tuples).
\end{rmk}

We may also deduce the following.

\begin{cor}\label{DateqDpp} \marginpar{DateqDpp}  If ${\cal D}$ is a definable subcategory then $M\in {\cal D}$ is strictly ${\cal D}$-atomic iff $M$ is locally ${\cal D}$-pure-projective.
\end{cor}
\begin{proof} ($\Rightarrow$)  Suppose that $M$ is strictly ${\cal D}$-atomic and $f:D\to M$ is a pure epimorphism in ${\cal D}$.  If $\overline{a}$ is a finite tuple from $M$, let $\phi$ pp be such that ${\rm pp}^M(\overline{a}) = \langle \phi\rangle_{\cal D}$.  By \ref{pureepichar}, there is $\overline{d}$ from $D$ with $f\overline{d} = \overline{a}$ and $\overline{d} \in \phi(D)$.  Since $M$ is strictly ${\cal D}$-atomic, there is $g:M\to D$ with $g\overline{a} = \overline{d}$, as required.

($\Leftarrow$)  By \ref{pureepi} there is a pure epimorphism $f: D' \to M$ in ${\cal D}$ with $D'$ strictly ${\cal D}$-atomic.  Now suppose that $\overline{a}$ is from $M$.  By assumption, there is $g:M \to D'$ such that $fg\overline{a}= \overline{a}$ and hence with ${\rm pp}^M(\overline{a}) = {\rm pp}^{D'}(g\overline{a})$.  Since $D'$ is ${\cal D}$-atomic, the latter pp-type is ${\cal D}$-finitely generated, by $\phi$ say.  Thus $M$ is ${\cal D}$-atomic.  

Now suppose that $D\in {\cal D}$ and $\overline{b} \in \phi(D)$.  By \ref{dsstrDat}, $D'$ is strictly ${\cal D}$-atomic, so there is $h: D' \to M$ with $h.g\overline{a} = \overline{b}$.  Thus we obtain the morphism $hg:M \to D$ with $hg.\overline{a} =\overline{b}$, and so see that $M$ is strictly ${\cal D}$-atomic.
\end{proof}

Note the following.

\begin{lemma}  Suppose that ${\cal D}$ is a definable subcategory of ${\rm Mod}\mbox{-}R$ and $A\in {\rm mod}\mbox{-}R$.  If $A$ has a ${\cal D}$-envelope, $f:A\to D$, then $D$ is strictly ${\cal D}$-atomic (and hence may be taken to be $D_A$).
\end{lemma}
\begin{proof} Choose some strictly ${\cal D}$-atomic preenvelope $A\to D_A$.  Since each of $D$, $D_A$ is a ${\cal D}$-preenvelope of $A$, there are morphisms $g:D\to D_A$ and $h:D_A \to D$ such that $hgf=f$.  Then $hg$ is an automorphism of $D$ and hence $D$ is a direct summand of $D_A$ so, \ref{dsstrDat}, is strictly ${\cal D}$-atomic.
\end{proof}

If $M$ is a module, $\overline{a} = (a_1, \dots, a_n)$ an $n$-tuple of elements from $M$ and $b\in M$, then we say that $b$ is {\bf definable} by a pp formula (in $M$) {\bf over} $\overline{a}$ if there is a pp formula $\psi(\overline{x},y)$ with $M\models \psi(\overline{a}, b)$ and with $b$ the unique solution in $M$ to $\psi(\overline{a},y)$, equivalently (since $\psi$ is pp, so $\psi(\overline{a},y)$ defines a coset of $\psi(\overline{0},y)$) with $\psi(\overline{0},M) = 0$ (see Section \ref{secmodth}).

One might ask whether, given $A\in {\rm mod}\mbox{-}R$, one may choose a strictly ${\cal D}$-atomic $A\to D_A$ such that every element of $D_A$ is definable over the image of $A$.  The example $R={\mathbb Z}$, $A={\mathbb Z}_2$ and ${\cal D}$ the class of injective ${\mathbb Z}$-modules shows that in general the answer is negative, since $D_A$ clearly has, as a direct summand, ${\mathbb Z}_{2^\infty}$ which has many automorphisms which fix its submodule $A$.  Here is another example, this time with $A=R$.

\begin{example}  Take $k$ any field, $R=k[X,Y]/(X,Y)^2$ and ${\cal D}$ the definable subcategory generated by the injective hull $E(k)$ of the unique simple module $k$ (so ${\cal D} = {\rm Inj}\mbox{-}R$).  The ${\cal D}$-envelope of $R$, which is the minimal choice of $D_A$, is $E(R)=E(k)\oplus E(k)$.  Let $a$ denote the image of $1_R$ in $E(R)$ and consider any element $b\in E(R)$ such that $bY=aX$.  The pp-type of $b$ is generated, modulo the definable category of injective $R$-modules, by the formula $yX=0 (\wedge \,\exists y \,( yY=xX))$ but this is also satisfied by, for instance, any element of the form $b+c$ where $c$ is in the socle of $E(R)$.  (Put more algebraically, there are non-identity automorphisms of $E(R)$ which fix $a$.)
\end{example}

\subsection{Constructing strictly atomic models}\label{secmakctble}

In this subsection we present a proof of \ref{Mak} for countable rings.  The proof was found in a discussion with Philipp Rothmaler.

{\bf We suppose, throughout this subsection, that the ring $R$ is countable}; this implies that there are just countably many (pp) formulas. 

\vspace{4pt}

Let ${\cal D}$ be a definable subcategory of ${\rm Mod}\mbox{-}R$.

Recall from Section \ref{secmodth} that a {\bf pp-pair} - denoted $\phi/\psi$ - is a pair $\phi(\overline{x}) \geq \psi(\overline{x})$ of pp formulas, where the inequality means that $\phi(M) \geq \psi(M)$ for all modules $M$.  Also, every definable category is determined by the set of pp-pairs which are closed on it, where we say that a pp-pair $\phi \geq \psi$ is {\bf closed} on $M$ if $\phi(M) = \psi(M)$  and is {\bf closed} on ${\cal D}$ if it is closed on $M$ for all $M\in {\cal D}$.  

For $\phi$ a pp formula, set $\phi_{\downarrow\cal D} = \{ \psi: \phi \geq \psi \text{ and }\phi/\psi \text{ is closed on } {\cal D}\}$ - a subset of $\langle \phi \rangle_{\cal D}$.

\begin{theorem}  Suppose that ${\cal D}$ is a definable subcategory of the module category ${\rm Mod}\mbox{-}R$ where $R$ is countable.  Let $A$ be a finitely presented $R$-module.  Then there is a ${\cal D}$-preenvelope $A\to D_A$ where $D_A$ is strictly ${\cal D}$-atomic.
\end{theorem} 
\begin{proof}
The construction of $D_A$ is an inductive one; set $B_0=A$.

\vspace{4pt}

Say $A$ is generated by $\overline{a} = \overline{a}_0$ with pp-type generated by the (quantifier-free) formula $\theta(\overline{x}_1)$.  Set $\theta_0 =\theta$.

Enumerate:  $(\theta_0)_{\downarrow\cal D} =\{ \phi_{1j}: j\geq 1\}$.

\vspace{4pt}

Let $\overline{a}_1$ in $B_1$ be a free realisation of $\phi_{11}$ so, by \cite[1.2.17]{PreNBK}, we have $f_0:A\to B_1$ with $f\overline{a}_0 =\overline{a}_1$.  Take $\overline{b_1}=\overline{a}_1\overline{b}_1'$ generating $B_1$, with pp-type generated by $\theta_1 = \theta_1(\overline{x}_1, \overline{x}_1')$; set $\overline{x}_2 = \overline{x}_1 \, \overline{x}_1'$ - the concatenation of $\overline{x}_1$ and $\overline{x}_1'$.

\vspace{4pt}

Enumerate:  $(\theta_1)_{\downarrow\cal D} = \{\phi_{2j}; j\geq 1\}$.

\vspace{4pt}

Let $\overline{a}_2$ in $B_2$ be a free realisation of $\phi_{12}(\overline{x}_1) \wedge\phi_{21}(\overline{x}_1, \overline{x}_1')$ and choose a morphism $f_1:B_1\to B_2$ taking $\overline{b}_1$ to $\overline{a}_2$.

\vspace{4pt}

Continue inductively:  having produced a free realisation $(B_n, \overline{a}_n)$ of $\phi_{1n}(\overline{x}_1) \wedge \phi_{2,n-1}(\overline{x}_2) \wedge \dots \wedge \phi_{n1} (\overline{x}_n)$, and a morphism $f_{n-1}:B_{n-1} \to B_n$ taking $\overline{b}_{n-1}$ to $\overline{a}_n$, choose a generating tuple $\overline{b_n} = \overline{a}_n \, \overline{b}_n'$ for $B_n$, with pp-type generated by $\theta_n = \theta_n(\overline{x}_{n+1}) = \theta_n (\overline{x}_n, \overline{x}_n')$.

\vspace{4pt}

Then enumerate $(\theta_n)_{\downarrow\cal D} = \{ \phi_{n+1,j}: j\geq 1\}$ and continue by choosing a free realisation $(B_{n+1}, \overline{a}_{n+1})$ of $\phi_{1,n+1}(\overline{x}_1) \wedge \phi_{2,n}(\overline{x}_2) \wedge \dots \wedge \phi_{n+1,1} (\overline{x}_{n+1})$, and a morphism $f_n:B_n \to B_{n+1}$ taking $\overline{b}_n$ to $\overline{a}_{n+1}$

\vspace{4pt}

Having continued the construction inductively, set $D_A= \varinjlim ((B_n)_n, (f_n:B_n \to B_{n+1})_n)$, with $f_{n\infty}:B_n \to D_A$ the limit maps.

\vspace{4pt}

We {\bf claim} that $D_A$ is in ${\cal D}$, is strictly ${\cal D}$-atomic and the pp-type of $f_{0\infty}(\overline{a})$ in $D_A$ is ${\cal D}$-generated by $\theta$.

\vspace{4pt}

Before going on to prove our claims, we note some points about the construction:

\noindent $\bullet$  for each $n$ and $m\geq n$, $f_{nm}\overline{b}_n$ is an initial segment of $\overline{b}_m$, where $f_{nm}$ denotes the composition $f_{m-1,m}\dots f_{n,n+1}$;

\noindent $\bullet$  for each $n$, the formula $\theta_n$ is ${\cal D}$-equivalent to each $\phi_{n+1,j}$ where, recall, we say that two formulas are ${\cal D}${\bf -equivalent} if they have the same solution set in each $D\in {\cal D}$;

\vspace{4pt}

\noindent and a lemma:

\vspace{4pt}

\begin{lemma}\label{5.4}\marginpar{5.4}  Given any $B_n$ and morphism $g:B_n \to D \in {\cal D}$, there is a factorisation through $f_n:B_n \to B_{n+1}$, and hence, by induction, through any $f_{nm}:B_n \to B_m$.
\end{lemma}
\begin{proof} of Lemma:  Since $\overline{a}_{n+1} =f_n\overline{b}_n$ is a free realisation of $\phi_{1,n+1}(\overline{x}_1) \wedge \phi_{2,n}(\overline{x}_2) \wedge \dots \wedge \phi_{n+1,1} (\overline{x}_{n+1})$, it will be sufficient to show that $g\overline{b}_n$ satisfies each of the formulas $\phi_{i, n+2-i}(\overline{x}_i)$.  But, for $i=1, \dots, n+1$, $\phi_{i, n+2-i}(\overline{x}_i) \in (\theta_{i-1})_{\downarrow\cal D}$ and $\overline{b}_{i-1}$ satisfies $\theta_{i-1}$, hence so does $gf_{i-1,n}\overline{b}_{i-1}$, which is the initial segment of $g\overline{b}_n$ from which we deduce that $g\overline{b}_n$ satisfies $\phi_{i,n+2-i}(\overline{x}_i)$.
\end{proof} of Lemma.

\begin{cor}\label{5.5}\marginpar{5.5} of Lemma:  If $\psi$ generates the pp-type of $f_{nm}(\overline{b}_n)$ in $B_m$, then $\psi \in \langle \theta_n\rangle_{\cal D}$.
\end{cor}
\begin{proof} of Corollary:  Suppose that $D\in {\cal D}$ and  $\overline{d} \in \theta_n(D)$.  So there is $g:B_n \to D$ taking $\overline{b}_n$ to $\overline{d}$.  By the Lemma, this extends to a morphism from $B_m$ to $D$.  So $\overline{d} \in \psi(D)$, as required.
\end{proof} of Corollary

Now the proofs of the claims:

\vspace{4pt}

\noindent 1)  $D_A\in {\cal D}$: $\,\,$ Suppose $\phi/\psi$ is closed on ${\cal D}$ and take $\overline{d} \in \phi(D_A)$.  Note that $\overline{d} = f_{n\infty}\overline{b}_n \cdot \overline{r}$ for some $n$ and matrix $\overline{r}$ over $R$.\footnote{If $d$ is a single element, this means just $d=\sum_i\, f_{n\infty}(b_i)r_i$, that is, $\overline{r}$ is a column vector; if $\overline{d} = (d_1,\dots d_k)$ then $\overline{r}$ is a matrix with $k$ columns.}  Since pp formulas commute with directed colimits \cite[1.2.31]{PreNBK}, we may take $n$ to be such that $\overline{b}_n\cdot \overline{r} \in \phi(\overline{x_n}\cdot \overline{r})(B_n)$.\footnote{Here $\phi$ is a formula with $k$ free variables and $\phi(\overline{x}_n \cdot \overline{r})$ is the pp formula where the $t$-th free variable is replaced by $\sum_i \, x_ir_{it}$, $t=1, \dots, k$.}  Therefore $\psi(\overline{x}_n\cdot \overline{r}) \in (\theta_n(\overline{x}_n))_{\downarrow\cal D}$, so is $\phi_{n+1,j}$ for some $j$.  By construction, there is $m\geq n$ (indeed $m=n+j$ works) such that $f_{nm}\overline{b}_n \cdot \overline{r} \in \psi(\overline{x}_n\cdot \overline{r})(B_m)$.  Since $\overline{d} = f_{n\infty}\overline{b}_n \cdot \overline{r} = f_{m\infty}f_{nm}\overline{b}_n \cdot \overline{r}$, we deduce that $\overline{d} \in \psi(D_A)$.

Thus every pp-pair closed on ${\cal D}$ is closed in $D_A$, hence, by definition of definable categories (see Section \ref{secmodth}) $D_A\in {\cal D}$.

\vspace{4pt}

\noindent 2)  $D_A$ is ${\cal D}$-atomic: $\,\,$ Suppose $\overline{d}$ is from $D_A$, say $\overline{d} = f_{n\infty}\overline{b}_n \cdot \overline{r}$ as above.  It is sufficient to take $\overline{d} =f_{n\infty}\overline{b}_n$ since, if the pp-type of the latter in $D_A$ is ${\cal D}$-generated by a pp formula $\theta$, then that of $f_{n \infty}\overline{b}_n \cdot \overline{r}$ will be ${\cal D}$-generated by $\exists \overline{y} \, (\theta(\overline{y}) \wedge \overline{x} = \overline{y} \cdot \overline{r})$.  We claim that, in fact, the pp-type of $f_{n\infty}\overline{b}_n$ in $D_A$ is ${\cal D}$-generated by $\theta_n$.  Since $\theta_n \in {\rm pp}^{B_n}(\overline{b}_n)$ certainly $\theta_n \in {\rm pp}^{D_A}(f_{n\infty}\overline{b}_n)$.  In the other direction, we have that ${\rm pp}^{D_A}(f_{n\infty}\overline{b}_n) = \bigcup_{m\geq n} \, {\rm pp}^{B_n}(f_{nm}\overline{b}_n)$ (by \cite[1.2.31]{PreNBK}) again).  By \ref{5.5}, each ${\rm pp}^{B_n}(f_{nm}\overline{b}_n)$ is a subset of $\langle\theta_n\rangle_{\cal D}$, and so we have that ${\rm pp}^{D_A}(f_{n\infty}\overline{b}_n) \subseteq \langle\theta_n \rangle_{\cal D}$, as required.

\vspace{4pt}

In particular, if we start with a pp formula $\phi$ and take a free realisation $(A,\overline{a}')$ of $\phi$ as the starting point of the construction, if we choose a generating tuple $\overline{a}_0 = \overline{a}' \, \overline{b}_0'$, then continue and build $D_A$ as before, then the pp-type of the image of $\overline{a}_0$ in $D_A$ will be ${\cal D}$-generated by $\phi$.  We state this for easy reference.

\begin{cor}\label{5.6} \marginpar{5.6} If $A$ is finitely presented and we construct $D_A$ as above, then for every $\overline{a}$ from $A$, if $\phi$ is such that $\langle \phi \rangle = {\rm pp}^A(\overline{a})$, then ${\rm pp}^{D_A}(f_{0\infty}\overline{a}) = \langle \phi \rangle_{\cal D}$.
\end{cor}

\vspace{4pt}

\noindent 3)  $D_A$ is strictly ${\cal D}$-atomic: $\,\,$ Since any tuple from $D_A$ has the form $f_{n\infty}\overline{b}_n \cdot \overline{r}$, it is enough to consider tuples of the form $f_{n\infty}\overline{b}_n $.  Suppose, then, that $\overline{d}$ is a tuple from $D\in {\cal D}$ such that ${\rm pp}^{D_A}(f_{n\infty}\overline{b}_n) \subseteq {\rm pp}^D(\overline{d})$.  We must produce a morphism from $D_A$ to $D$ extending the partial map which takes $f_{n\infty}\overline{b}_n$ to $\overline{d}$.  It will be sufficient, by construction of $D_A$, to extend inductively, so coherently, to each $B_m$, $m\geq n$, the morphism $g_n:B_n \to D$ defined by $\overline{b}_n \mapsto f_{n\infty}\overline{b}_n \mapsto \overline{d}$.

But that is exactly what \ref{5.4} above gives us.
\end{proof}

Remark:  Of course, the same applies to each $B_n$ in the construction (just take $B_n$ as the starting point).

\begin{cor} (cf.~\cite[3.13]{RotHab}) If $D_A$ is constructed as above then $D_A$ is a union of submodules $B'_n (= f_{n \infty}B_n)$ such that, if $\overline{b}'_n$ is a finite generating tuple for $B'_n$ and if $\theta'_n$ is a quantifier-free formula generating ${\rm pp}^{B'_n}(\overline{b}'_n)$, then ${\rm pp}^{D_A}(\overline{b}'_n)$ is ${\cal D}$-generated by $\theta'_n$.
\end{cor}

It would be nice to have a proof in somewhat the same style for arbitrary rings, either by making a ``wider" construction that still makes only countably many construction steps (Makkai's proof, building a term model, is in this style but is not very ``algebraic") or using the fact that an uncountable ring is a direct union of elementary subrings of smaller (even countable) cardinality, and then tensoring up what we have over those subrings (results from \cite{PreTens} and \cite{PreRepEE} describe what happens when tensoring up).  The latter approach would be neat but I have not seen a way to build coherent directed systems of preenvelopes over approximating definable subcategories.

\vspace{4pt}

We now drop the countability hypothesis on the ring $R$ which was imposed in this subsection.

\subsection{The ring of definable scalars}

If ${\cal D}$ is a definable subcategory of ${\rm Mod}\mbox{-}R$, then the {\bf ring} $R_{\cal D}$ {\bf of definable scalars} of ${\cal D}$ is the set of pp-definable maps on ${\cal D}$ (see Section \ref{secmodth}); if ${\cal D} = \langle M \rangle$, the definable subcategory generated by $M$, we also write $R_M$ for $R_{\cal D}$.

\begin{lemma}\label{cyclic} \marginpar{cycli}  If $R\xrightarrow{f}D_R$ is any ${\cal D}$-preenvelope of $R$ in ${\cal D}$, then $D_R$ is cyclic, generated by $a=f1$, over its endomorphism ring.
\end{lemma}
\begin{proof}
If $b\in D_R$ and $g:R \to D_R$ is defined by $1\mapsto b$, then the preenveloping property gives us an endomorphism of $D_R$ as shown in the diagram and as required.
$\xymatrix{R \ar[r]^f \ar[rd]_g & D_R\ar[d] \\ & D_R}$ 
\end{proof}

\begin{lemma}\label{rds} \marginpar{rds}  If $R\xrightarrow{f}D_R$ with $a=f1$ is any ${\cal D}$-atomic preenvelope of $R$ in ${\cal D}$ and if $R'$ is the ring of definable scalars of $D_R$, then $aR'$ is the submodule consisting of those elements which are definable in $D_R$ by a pp formula with parameter $a$:  
$$aR' = \{ b\in D_R: \mbox{ if } {\rm pp}^{D_R}(a,b) = \langle \phi\rangle_{\cal D} \mbox{ then } \phi(0,D_R)=0  \}.$$
\end{lemma}
\begin{proof}  
The image of $a$ under any definable map is definable over $a$, so certainly any element in $aR'$ is definable over $a$.

If $\phi$ ${\cal D}$-generates the pp-type of $(a,b)$ then since, for every $c\in D_R$ there is an endomorphism $f$ of $D_R$ taking $a$ to $c$, and hence with $\phi(c,fb)$, we see that $\phi$ defines a total relation on $D_R$.  Therefore $\phi$ gives a definable map on $D_R$ iff it is functional on $D_R$, that is, iff $\phi(0,D_R)=0$, giving the second statement.
\end{proof}

\begin{prop}\label{rdsbiend} \marginpar{rdsbiend}  Suppose that ${\cal D}$ is a definable subcategory of ${\rm Mod}\mbox{-}R$ and that $M\in {\cal D}$ is strictly ${\cal D}$-atomic and is finitely generated over its endomorphism ring.  Set $R_M$ to be the ring of definable scalars of $M$.  Then $R_M = {\rm Biend}(M_R)$ - the biendomorphism ring ${\rm End}_{{\rm End}(M_R)}(M)$ of $M_R$.
\end{prop}
\begin{proof} The proof of \cite[6.1.19]{PreNBK} works in this situation; we essentially repeat it here.

Let $ g\in {\rm Biend}(M_R)$: it must
be shown that the action of $ g $ on $ M $ is pp-definable in $
M_R. $ Set $ S={\rm End}(M_R) $ and suppose that $ a_1,\dots
,a_k\in M $ are such that $ _SM=\sum _1^kSa_i. $ Then $ g $ is
determined by its action on $ \overline{a}=(a_1,\dots ,a_k) $, so
consider $ g\overline{a}. $ Since $ M $ is strictly ${\cal D}$-atomic, the pp-type of $(\overline{a}, g\overline{a})$ is ${\cal D}$-finitely generated, by, say, $\phi$.

Consider the pp formula $ \rho (u,v) $ which is 
$$ \exists x_1,\dots ,x_k \,\exists y_1,\dots ,y_k \, \big(u=\sum_1^kx_i \,\, \wedge \,\, v=\sum _1^ky_i \,\, \wedge\,\, \overline{\phi}(\overline{x}, \overline{y})\big)$$
where $\overline{\phi}(x_1,\dots, x_k, y_1,\dots, y_k)$ is $\bigwedge_{i=1}^k \phi_i(x_i,y_i)$ where $\phi_i(x_i,y_i)$ is 
$$\exists z_{i1},\dots,\hat{z}_{ii},\dots ,z_{ik}, w_{i1},\dots ,\hat{w}_{ii}, \dots
,w_{ik} \,\, \phi (z_{i1},\dots ,x_i,\dots ,z_{ik}, w_{i1},\dots,y_i,\dots ,w_{ik}).$$

It follows directly, from the strict ${\cal D}$-atomic condition and choice of $\phi$, that $M\models \phi_i(c,d)$ iff there is $s\in S$ with $sa_i=c$ and $sa_ig=d$ (the formula $\phi_i(c,d)$ says that $c$, $d$ are the $i$-th components of tuples satisfying $\phi$; note that such tuples are exactly the images of $\overline{a}$ and $g\overline{a}$ under (the same) endomorphisms).  In particular, for each $i$ and $s$, we have $M\models \phi_i(sa_i, sa_ig)$.  We claim that $\rho$ defines the action of $g$ in $M$.

First, $ \rho (u,v) $ defines a total relation  from $ u $ to $
v$: given $ c\in M $ we have $ c=\sum _1^ks_ia_i $ for some $
s_i\in S $, hence $ cg=\sum _1^ks_ia_ig. $  As commented above, $\bigwedge_{i=1}^k \phi_i(s_ia_i, s_ia_ig)$, holds.  Therefore $ (c,cg)\in \rho(M)$.

It remains to show that $ \rho  $ is functional, so suppose
$ (0,d)\in \rho(M) $.  Then there are $ c_i, d_i\in M $ such
that $ 0=\sum _1^kc_i$, $ d=\sum _1^kd_i $ and such that $M\models \phi_i(c_i,d_i)$ for each $ i$.  As commented above, it follows that there are $ s_i\in S $ with $
s_ia_i=c_i $ and $ s_ia_ig=d_i $ for
$ i=1,\dots ,k. $  So $ d=\sum d_i=\sum s_ia_ig=(\sum
s_ia_i)g $ and $ 0=\sum c_i=\sum s_ia_i, $ from which we deduce $
d=0, $ as required.
\end{proof}

\begin{cor}\label{Denvrds} \marginpar{Denvrds}  Suppose that $R\to D_R$ is a strictly ${\cal D}$-atomic ${\cal D}$-preenvelope of $R$.  Then the ring,  $R_{D_R}$, of definable scalars of $D_R$ is ${\rm Biend}(D_R)$.
\end{cor}
\begin{proof}  This is immediate from \ref{cyclic} and \ref{rdsbiend} but here is a simpler direct proof.  

Every definable scalar of $D_R$ is a biendomorphism so, for the converse, take $\alpha \in {\rm Biend}(M)$ and set $b=a\alpha$, where $a$ is the image of $1_R$ in $D_R$.  Choose a generator $\rho$ for ${\rm pp}^{D_R}(a,b)$; we claim that $\rho$ defines a scalar on $D_R$. Since $D_R$ is generated by $a$ as an ${\rm End}(D_R)$-module, $\rho$ is total on $D_R$.  Also, if we have $\rho(0,d)$ for some $d\in D_R$, then there is an endomorphism $f$ of $D_R$ with $fa=0$ and $fb=d$.  But then $d=fb =f(a\alpha) = (fa)\alpha =0$, as required.
\end{proof}

\section{Tilting and silting classes}\label{sectilt} \marginpar{sectilt}

An $R$-module $T$ is {\bf tilting} if ${\rm Gen}(T)=T^{\perp_1}$.  Here ${\rm Gen}(T)$ is the class of modules $M$ generated by $T$, that is, there is an epimorphism $T_0 \to M$ with $T_0\in {\rm Add}(T)$, where ${\rm Add}(T)$ is the closure of $T$ under direct sums and direct summands, and $T^{\perp_1} = \{ M: {\rm Ext}^1(T,M)=0\}$.  Equivalently $T$ is tilting iff ${\rm pdim}(T) \leq 1$, if ${\rm Ext}^1(T,T^{(\kappa)})=0$ for any $\kappa$ and if there is an exact sequence $0 \to R \to T_0 \to T_1 \to 0$ with $T_0, T_1 \in {\rm Add}(T)$.  If so, then $R\to T_0$ is a ${\rm Gen}(T)$-preenvelope of $R$.  Also, if $T$ is tilting, then ${\rm Gen}(T) = {\rm Pres}(T)$ (the class of modules presented by ${\rm Add}(T)$), that is, for any $M \in {\rm Gen}(T)$, there is an exact sequence $T_1 \to T_0 \to M \to 0$ with $T_0, T_1 \in {\rm Add}(T)$.

More generally an $R$-module $T$ is {\bf silting} if the class ${\rm Gen}(T)$ that it generates has the form ${\cal D}_\sigma = \{M: \text{ the map }(\sigma,M)\,:\,(P_1,M) \to (P_0,M) \text{ is surjective}\}$ for some projective presentation $P_0 \to P_1 \to T \to 0$ of $T$, in which case we refer to this as the {\bf silting class}  that it generates; note that it is a definable subcategory of ${\rm Mod}\mbox{-}R$.  Furthermore, the elementary dual definable category of ${\rm Gen}(T)$ is that, ${\rm Cogen}(T^\ast)$, cogenerated by the dual cosilting module $T^\ast$.

For all this, including the fact that cosilting modules are pure-injective, see, for instance, \cite{ColFul}, \cite{AHTT}, \cite{AHMVSilt}.

Also (\cite[9.8]{AngHerML}), if $T$ is tilting, then ${\rm Add}(T) \subseteq {\rm Cogen}(T^\ast)\mbox{-}ML$, that is, every module in ${\rm Add}(T)$ is ${\rm Gen}(T)$-atomic.  In fact, we have the following where, for tilting modules this is, by \ref{streqstr}, already contained in \cite[9.8]{AngHerML}.

\begin{prop}\label{siltstrDat} \marginpar{siltstrDat}  Suppose that $T$ is a silting module in ${\rm Mod}\mbox{-}R$ and let ${\cal D} = {\rm Gen}(T)$ be the (definable) silting class generated by $T$.  
Then $T$ is strictly ${\cal D}$-atomic and, for some $n$, $R\to T^n$ is a ${\cal D}$-preenvelope for $R$.  The same is true for ${\cal D} = \langle T \rangle$, the definable category generated by $T$.
\end{prop}
\begin{proof} Let ${\cal D}$ be either ${\rm Gen}(T)$ or $\langle T \rangle$.  By \ref{pureepi} there is a pure epimorphism $p: D \to T$ with $D \in {\cal D}$ strictly ${\cal D}$-atomic, so we get an exact sequence $ 0 \to K = {\rm ker}(p) \to D \to T \to 0$ with $K \in {\cal D}$.  Now, $T$ is ${\rm Ext}$-projective in ${\cal D}$ (\cite[3.8]{AHMVSilt}), so $T$ is a direct summand of $D$ hence, by \ref{dsstrDat}, $T$ is strictly ${\cal D}$-atomic, as claimed.

Let $R \to D_R$ be a strictly ${\cal D}$-atomic ${\cal D}$-preenvelope of $R$.  Let $a$ be the image of $1$ in $D_R$.  Since $D_R \in {\cal D}\subseteq {\rm Gen}(T)$, there is an epimorphism $g:T^{(\kappa)} \to D_R$.  Since $D_R$ is locally ${\cal D}$-projective, \ref{DateqDpp}, there is $h:D_R \to T^{(\kappa)}$ with $gha=a$ and hence such that the pp-type of $ha$ in $T^{(\kappa)}$, and hence in $T^n$ for some $n$, is ${\cal D}$-generated by the pp formula $x=x$.  Therefore the composition $R \to D_R \xrightarrow{h} T^{(\kappa)} \xrightarrow{\pi} T^n$ is a ${\cal D}$-preenvelope of $R$.
\end{proof}

By \ref{dsstrDat} we deduce the following.

\begin{cor}\label{siltstrat} \marginpar{siltstrat} Suppose that $T$ is a silting module in ${\rm Mod}\mbox{-}R$ and let ${\cal D} = {\rm Gen}(T)$ be the (definable) silting class generated by $T$ or ${\cal D} = \langle T \rangle$ the definable subcategory generated by $T$.  Then every module in ${\rm Add}(T)$ is strictly ${\cal D}$-atomic.
\end{cor}

The converse - that every strictly ${\cal D}$-atomic module be in ${\rm Add}(T)$ - is far from true.  For ${\cal D} = {\rm Gen}(T)$ take $T=R$:  in general not every finitely presented $R$-module is projective.  For ${\cal D} = \langle T \rangle$ take $T={\mathbb Z}_{2^\infty}$.  Every module in $\langle {\mathbb Z}_{2^\infty}\rangle$ is $\Sigma$-pure-injective so is strictly $\langle {\mathbb Z}_{2^\infty}\rangle$-atomic (by \cite[4.4.7, 4.3.9]{PreNBK}) but ${\mathbb Q} \in \langle {\mathbb Z}_{2^\infty}\rangle$ is not in ${\rm Add}( {\mathbb Z}_{2^\infty})$.  In \ref{strDrat} we do see a special case where every strictly ${\rm Gen}(T)$-atomic module is pure in a direct sum of copies of $T$.

The next result now follows from \ref{dualnegisol}.

\begin{cor}\label{cosiltnegisol} \marginpar{cosiltnegisol}  Suppose that $T$ is a silting module in ${\rm Mod}\mbox{-}R$, $S={\rm End}(T)$ and $E$ is a minimal injective cogenerator for $S\mbox{-}{\rm Mod}$.  Set $T^\ast = {\rm Hom}_S(T,E)$ to be the dual cosilting module.  Then every indecomposable pure-injective direct summand of $T^\ast$ is neg-isolated with respect to the definable class $\langle T^\star \rangle \subseteq {\rm cogen}(T^\ast)$ generated by $T^\ast$.
\end{cor}

Since, \cite[1.2]{AHTT}, any tilting, hence any silting, module $T$ is finitely generated over its endomorphism ring, \ref{rdsbiend} applies to $T$.

\begin{cor} Let $T$ be a silting $R$-module and set $R_T$ to be its ring of definable scalars.  Then $R_T = {\rm Biend}(T_R)$ (as $R$-algebras).  
\end{cor}

If $T$ is a tilting module, then (\cite[2.1]{AHTT}) every module $M$ has a {\bf special} ${\rm Gen}(T)$-{\bf preenvelope}, that is, there is an exact sequence $0 \to M \xrightarrow{i} T_0 \to T_1 \to 0$ with $i$ a ${\rm Gen}(T)$-preenvelope of $M$ and $T_1 \, \in \, ^{\perp_1}{\rm Gen}(T)$, that is ${\rm Ext}^1(T_1,{\rm Gen}(T))=0$.

\begin{lemma}\label{specenv} \marginpar{specenv}  Suppose that $T$ is a tilting $R$-module and $M\in {\rm Mod}\mbox{-}R$ is such that ${\rm Ext}^1(M,{\rm Gen}(T))=0$.  Then there is an exact sequence $0 \to M \xrightarrow{i} T_0 \to T_1 \to 0$ with $T_0, T_1 \in {\rm Gen}(T)$, $i$ a ${\rm Gen}(T)$-preenvelope of $M$ and ${\rm Ext}^1(T_0,{\rm Gen}(T))=0 ={\rm Ext}^1(T_1,{\rm Gen}(T))$.  It follows that $T_0, T_1 \in {\rm Add}(T)$

Furthermore, in the case $M=R$, $T_0 \oplus T_1$ is a tilting module equivalent to $T$.
\end{lemma}
\begin{proof}
The first statement follows from the proof of \cite[1.2]{AHTT}, alternatively see (the proof of) \cite[13.18]{GT}.  The fact that ${\rm Gen}(T) \cap \,^{\perp_1}{\rm Gen}(T) = {\rm Add}(T)$ is \cite[13.10(c)]{GT}.  The last comment is \cite[13.19]{GT}. 
\end{proof}

In the next section we focus on a special case.

\section{Modules of irrational slope}\label{secirrat} \marginpar{secirrat}

We suppose throughout Section \ref{secirrat} that $R$ is a tubular algebra.  For these algebras and their finite-dimensional modules see  \cite[Chpt.~5]{RinTame} or any of the references cited below.  Our eventual aim is to complete the description of the infinite-dimensional indecomposable pure-injective modules which was begun in \cite{HarThes}, \cite{HarPre} and continued in \cite{GreTub} and \cite{KusLak}.  The task which remains is to describe the modules of irrational slope.  Here we make a little progress in this direction.

We refer to \cite{HarPre}, \cite{ReiRin} for what we need on the modules and morphisms between them, and to \cite{AngKusAlg} for tilting and cotilting modules over these algebras.  We do recall that every finite-dimensional indecomposable module has a well-defined {\em slope}, which is a nonnegative rational number\footnote{except for those in certain components of preprojective or preinjective modules, but we can ignore these modules here} or $\infty$ and that there is only the zero morphism from a module of slope $r$ to a module of slope $s<r$.  We say that a direct sum of indecomposable modules all of slope $r$ is also of slope $r$.

This terminology extends to certain infinite-dimensional modules.  Let $r$ be a positive real.  Denote by ${\bf p}_r$ the finite-dimensional indecomposable modules of slope $<r$ and by ${\bf q}_r$ those of slope $>r$.  We will also use these notations for their ${\rm add}$-closures (i.e.~close under finite direct sums).  Let ${\cal C}_r = {\bf q}_r^{\perp_0}$, ${\cal B}_r = \, ^{\perp_0}{\bf p}_r$ and set ${\cal D}_r ={\cal B}_r \, \cap {\cal C}_r$.  This is the category of {\bf modules of slope} $r$ and it is a definable subcategory of ${\rm Mod}\mbox{-}R$.  It is the case that {\em every} indecomposable module, finite- or infinite-dimensional, has a slope \cite[13.1]{ReiRin}.  The categories ${\cal D}_r$ with $r$ rational are well understood but little is known in the case that $r$ is irrational.  Note that, if $r$ is irrational, the only finite-dimensional module in ${\cal D}_r$ is $0$.  The category ${\cal D}_r$ is closed under extensions:  if $0 \to D \to X \to D' \to 0$ is an exact sequence with $D, D' \in {\cal D}_r$, then we have $({\bf q}_r,X)=0 = (X, {\bf p}_r)$ and hence $X\in {\cal D}_r$.  In fact, we will see \ref{espes} below that, if $r$ is irrational, then every exact sequence in ${\cal D}_r$ is pure-exact.

\begin{rmk}  Every exact sequence being pure-exact is a property of the category of modules over a von Neumann regular ring but ${\cal D}_r$ is not an abelian category:  there are epimorphisms between modules in ${\cal D}_r$ whose kernel is not in ${\cal D}_r$ (and which are not the cokernel of any kernel in ${\cal D}_r$); similarly for some monomorphisms in ${\cal D}_r$.  But we do say something, see \ref{fgker} and \ref{kerinDr}, about the non-pure morphisms in ${\cal D}_r$. 
\end{rmk}

See \cite[\S\S 2, 3]{HarPre}, corrected and extended in \cite[\S 7]{GreTub}, for more about slopes and supports.  In particular we have the following.

\begin{prop}\label{slope} (\cite[3.4]{HarPre}\footnote{The reference for \cite[3.3]{HarPre}, which is used to prove this result, should read Lemma 13.4 (not Lemma 11) of Reiten and Ringel's paper.} Suppose that $M$ is a module of slope $r>0$.  Then, for every $\epsilon >0$, $M$ is a directed union of finite-dimensional submodules with slopes in the interval $(r-\epsilon, r]$ (indeed in $(r-\epsilon, r)$ in the case that $r$ is irrational).
\end{prop}

We say that a module $M$ is {\bf supported} on an interval $I$ of the extended real line if $M$ is a directed sum = directed union of finitely generated submodules whose indecomposable direct summands have slope in $I$.

Recall \cite[6.4]{AngKusAlg} that there is a unique to ${\rm Add}$-equivalence tilting module $T$ in ${\cal D}_r$ and a unique to ${\rm Prod}$-equivalence cotilting module $C$ in ${\cal D}_r$ and so (\cite{AHTT}, \cite{AngKusAlg}) ${\cal B}_r = {\rm Gen}(T) = {\rm Pres}(T)$, ${\cal C}_r = {\rm Cogen}(C) = {\rm Copres}(W)$, hence ${\cal D}_r = {\rm Gen}(T)\, \cap\, {\rm Copres}(C)$.  Also recall \cite{AHTT} that the partial tilting modules - the modules $T' \in {\rm Add}(T)$ - are ${\rm Ext}$-projectives in ${\cal D}_r$, meaning that ${\rm Ext}^1(T',{\cal D}_r)=0$ and hence that any exact sequence $0 \to D' \to D \to T' \to 0$ with $D', D \in {\cal D}_r$ splits.  Dually, the partial cotilting modules - the modules $C' \in {\rm Cogen}(C)$ - are Ext-injectives in ${\cal D}_r$:  ${\rm Ext}^1({\cal D}_r, C')=0$ and every exact sequence $0 \to C' \to D \to D'' \to 0$ with $D, D'' \in {\cal D}_r$ splits.  If $T$ is a tilting module, then $T^\ast = {\rm Hom}_k(T,k)$ is, \cite[3.4]{AngHrb}, a cotilting module for the dual definable category $({\cal D}_r)^{\rm d}$, which is (the duality takes an irrational cut on indecomposable right modules to an irrational cut on indecomposable left modules) the category of left $R$-modules of some irrational slope $r^\ast$.  So, by left/right symmetry, the cotilting module $C$ for ${\cal D}_r$ may be taken to be the dual (in this sense) module for some tilting left $R$-module which belongs to the category of left modules of slope $r^\ast$.

\begin{lemma}\label{tiltapprox} \marginpar{tiltapprox} Let $T$ be a tilting module of irrational slope $r$.  For every $A\in {\rm mod}\mbox{-}R$ supported on $(-\infty,r)$ there is an exact sequence $0 \to A \xrightarrow{i} T_0 \to T_1 \to 0$ with $T_0, T_1 \in {\rm Add}(T)$ and $i$ a ${\cal D}_r$-preenvelope of $A$.
\end{lemma}
\begin{proof}  We check the criterion of \ref{specenv} to show that there is such an exact sequence with $i$ a ${\rm Gen}(T)$-preenvelope of $A$ and hence, since $T_0, T_1 \in {\cal D}_r \subseteq {\rm Gen}(T)$, a ${\cal D}_r$-preenvelope of $A$.

First, we recall (see \cite[p. 697]{HarPre}) that ${\cal B}_r = {\bf p}_r^{\perp_1}$ so we have ${\rm Ext}^1(A, {\cal D}_r) =0$.  In particular ${\rm Ext}^1(A, T) =0$ and hence, since, by \cite[3.1.5]{RinTame}, every finite-dimensional module with positive rational slope has projective dimension 1, in particular ${\rm Ext}^2(A,-)=0$, we have ${\rm Ext}^1(A, {\rm Gen}(T)) =0$.  Therefore \ref{specenv} applies. 
\end{proof}

\begin{cor}\label{dualtiltr} \marginpar{dualtiltr}  Let $T$ be a tilting module of irrational slope $r$ and let $T^\ast$ be its dual, cotilting module, of irrational slope $r^\ast$.  Then every indecomposable direct summand of $T^\ast$ is neg-isolated with respect to ${\cal D}_{r^\ast}$.
\end{cor}
\begin{proof}  This is by \ref{cosiltnegisol} and since the definable subcategory generated by $T^\ast$ is, \cite[8.5]{HarPre}, all of ${\cal D}_{r^\ast}$.
\end{proof}

We know \cite[7.4, 7.5]{HarPre}, at least if $R$ is countable, that there are superdecomposable pure-injectives in ${\cal D}_{r^\ast}$ so it {\em might} be that $T^\ast$ has superdecomposable direct summands.

Reversing the roles of $r$ and $r^\ast$, we deduce the following (which can be obtained by other arguments, see \cite{KusLak}).

\begin{cor} If $r$ is an irrational then there is a cotilting module of slope $r$ all of whose indecomposable direct summands are neg-isolated in ${\cal D}_r$.
\end{cor}

Indeed, applying \cite[3.5]{MehPre}, there is such a cotilting module with no superdecomposable direct summand, hence which is the pure-injective hull of a direct sum of neg-isolated (in particular, indecomposable) pure-injectives.

\subsection{Pp formulas near an irrational and definable closures}

The first parts of the next result come from the thesis \cite{HarThes} of Harland, see \cite[3.2]{HarPre}.  In fact, the result in \cite{HarPre} is for formulas in one free variable/single elements, but the proof works just as well for finite tuples (the change is essentially notational).

\begin{theorem}\label{localpp} \marginpar{localpp}  Let $r$ be a positive irrational and let $\phi(\overline{x})$ be a pp formula for $R$-modules.  Then there is a pp formula $\phi'$, a free realisation $(C',\overline{c}')$ of $\phi'$ and $\epsilon >0$ such that:

\noindent  (1) $C' \in {\rm add}({\bf p}_{r-\epsilon})$;

\noindent  (2) $C'/\langle \overline{c}'\rangle \in {\rm add}({\bf q}_{r+\epsilon})$;

\noindent  (3) $\phi(X) = \phi'(X)$ for every $X$ supported on $(r-\epsilon, r+\epsilon)$

\noindent  (4) each morphism $C'\to X$ with $X$ supported on $(r-\epsilon, r+\epsilon)$ is determined by $f\overline{c}'$

\noindent  (5) we may choose the formula $\phi'(\overline{x})$, say $\exists \overline{y} \, \theta'(\overline{x}, \overline{y})$ with $\theta'$ quantifier-free, such that there is a unique tuple $\overline{d}'$ from $C'$ such that $C'\models \theta'(\overline{c}', \overline{d}')$;

\noindent  (6) with $\phi'$, being $\exists \overline{y} \, \theta'(\overline{x}, \overline{y})$, chosen as in (5) above, if $X$ is supported on $(r-\epsilon, r+\epsilon)$ and $\overline{a} \in \phi(X) =\phi'(X)$, then there is a unique tuple $\overline{b}$ from $X$ with $(\overline{a}, \overline{b}) \in \theta'(X)$.
\end{theorem}
\begin{proof} (1)-(3) One just follows through the proof of \cite[3.2]{HarPre}, checking that it works for $n$-tuples in place of elements.

(4)  Suppose that $f, g:C' \to X$ are such that $f\overline{c}' = g\overline{c}'$.  Then $f-g$ factors through $C/\langle \overline{c}'\rangle$ which is supported on $[r+\epsilon, \infty)$, hence $f-g=0$.

(5)  Having made an initial choice of $\phi'$ being, say, $\exists \overline{y} \, \theta''(\overline{x}, \overline{y})$, choose $\overline{d}$ from $C'$ such that $C'\models \theta''(\overline{c}', \overline{d})$, then just replace $\theta''$ by a pp formula $\theta'$ which generates the pp-type of $\overline{c}' \overline{d}$ in $C'$ (using that the pp-type of any finite tuple in a finitely presented module is finitely generated), so we have $C'\models \theta'(\overline{c}', \overline{d})$.

Then, if there were another witness in $C'$ to the existential quantifiers in $\exists \overline{y} \, \theta'(\overline{c}', \overline{y})$, say $C'\models \theta'(\overline{c}', \overline{e})$, there would be $f:C'\to C'$ with $f\overline{c}' = \overline{c}'$ and $f\overline{d} = \overline{e}$.  But then $1-f:C'\to C'$ would factor through $C'/\langle \overline{c}' \rangle$, a contradiction as above.

(6)  We have that if $X\models \theta'(\overline{a}, \overline{b})$ and $X\models \theta'(\overline{a}, \overline{b}')$, then there are morphisms $f,f':C' \to X$ with $f:\overline{c}' \overline{d} \mapsto \overline{a} \overline{b}$ (where $\overline{d}$ is as in part 5) and $f': \overline{c}' \overline{d} \mapsto \overline{a} \overline{b}'$.  Then $f-f'$ factors through $C'/\langle \overline{c}' \rangle$ and so, as before, must be the zero map and $\overline{b} = \overline{b}'$, as claimed.
\end{proof}

Note that (6) says that, given a pp formula, there is a pp formula to which it is equivalent on every module supported near $r$ and which has, on any such module, unique witnesses to its existential quantifiers.

\vspace{4pt}

We now show that, if $D$ is a module of irrational slope $r$ and $\overline{a}$ is a tuple from $D$, then the pp-type of $\overline{a}$ is determined, within the category ${\cal D}_r$, by its pp-type in its definable closure, ${\rm dcl}^D(\overline{a})$, in $D$.

Recall, see Section \ref{secmodth}, that the {\bf definable closure}, ${\rm dcl}^D(A)$ or ${\rm dcl}^D(\overline{a})$ of a subset $A$ of, or tuple $\overline{a}$ in, $D$ means the set of elements $b\in D$ which are pp-definable in $D$ over $\overline{a}$.  This is a submodule of $D$.  Also, just from the definition of definable closure and the fact that morphisms preserve pp formulas, if $f,g:D \to D' \in {\cal D}_r$ agree on $\overline{a}$, then they agree on ${\rm dcl}^D(\overline{a})$.

\begin{cor}\label{ppdcl} \marginpar{ppdcl}  If $D\in {\cal D}_r$ and $\overline{a}$ is from $D$, then ${\rm pp}^D(\overline{a})$ is generated, modulo the definable category generated by ${\cal D}_r$, by ${\rm pp}^{{\rm dcl}^D(\overline{a})}(\overline{a})$.  That is, ${\rm pp}^D(\overline{a}) = _{{\cal D}_r} \, {\rm pp}^{{\rm dcl}^D(\overline{a})} (\overline{a})$.
\end{cor}
\begin{proof}  Suppose that $D\models \phi(\overline{a})$.  Choose $\phi'$ as in \ref{localpp}(5); say $\phi'(\overline{x})$ is $\exists \overline{z} \, \theta(\overline{x}, \overline{z})$.  So $D \models \phi'(\overline{a})$; say $D\models \theta(\overline{a}, \overline{b})$.  By \ref{localpp}(6), $\overline{b}$ is the unique solution to $\theta(\overline{a}, \overline{z})$ in $D$, so each component of the tuple $\overline{b}$ is definable in $D$ over $\overline{a}$.  Hence ${\rm dcl}^D(\overline{a}) \models \phi'(\overline{a})$ and so $\phi' \in {\rm pp}^{{\rm dcl}^D(\overline{a})} (\overline{a})$.  But $\phi$ and $\phi'$ are equivalent modulo the definable category generated by ${\cal D}_r$, as required.
\end{proof}

Note, see the example below, that this does {\em not} imply that the inclusion of ${\rm dcl}^D(\overline{a})$ in $D$ is pure, nor that ${\rm dcl}^D(\overline{a})$ is in ${\cal D}_r$.  That is, if $\overline{a}$ satisfies some pp formula $\phi$ in $D\in {\cal D}_r$, it need not be the case that there will be witnesses to the existential quantifiers of $\phi$ which are definable over $\overline{a}$; rather, there is some pp formula $\phi'$ with $\phi'(D) =\phi(D)$ for which there are definable-over-$\overline{a}$ witnesses to any existential quantifiers that $\phi'$ may have.

In particular, consider the case $M=R$ and a corresponding exact sequence $0 \to R \xrightarrow{f} T_0 \to T_1 \to 0$ as in \ref{tiltapprox}.  Set $a=f1$.  Then the pp-type of $a$ in $T_0$ is equivalent (by \ref{Denvpp}), modulo $\langle {\cal D}_r \rangle$, to the formula $x=x$ which generates ${\rm pp}^R(1)$.  Thus every formula $\phi$ such that $T_0 \models \phi(a)$ is equivalent, modulo the theory of ${\cal D}_r$, to $x=x$.  But certainly there will be such formulas which are not quantifier-free and which are not themselves witnessed in the definable closure (which by \ref{dcltriv} below is $aR$) of $a$ in $T_0$ - rather each is ${\cal D}_r$-equivalent to a formula (in this case $x=x$) which is so witnessed.

\begin{lemma}\label{dcltriv} \marginpar{dcltriv} Suppose that the module $M$ is supported on $(-\infty, r-\eta)$ for some $\eta >0$ and take an exact sequence $0 \to M \to T_0 \xrightarrow{p} T_1 \to 0$ with $T_0, T_1$ both of slope $r$.  Then ${\rm dcl}^{T_0}(M) = M$.
\end{lemma}
\begin{proof}  Suppose that $b\in {\rm dcl}^{T_0}(M)$, say $T_0 \models \rho(\overline{a}, b)$ for some pp formula $\rho$ with $\rho(\overline{0}, T_0) =0$ and with $\overline{a}$ from $M$.  Then $T_1 \models \rho(\overline{0}, pb)$ and so, since $T_1$ and $T_0$ generate the same definable category, see \ref{nodefsub} below, we deduce that $pb =0$ and $b\in M$, as claimed.
\end{proof}

\subsection{Purity in ${\cal D}_r$}

We have already referred to the fact that the category ${\cal D}_r$ has no non-zero proper definable subcategory.

\begin{theorem}\label{nodefsub} \marginpar{nodefsub} (\cite[8.5]{HarPre})    If $M, N\in {\cal D}_r$ are nonzero, then $M$ and $N$ are elementarily equivalent, in particular they open the same pp-pairs.  Hence, if $M\in {\cal D}_r$ is nonzero, then the definable subcategory $\langle M \rangle$ of ${\rm Mod}\mbox{-}R$ generated by $M$ is ${\cal D}_r$.
\end{theorem}

\begin{theorem}\label{espes} \marginpar{espes} Every exact sequence in ${\cal D}_r$ is pure-exact.
\end{theorem}
\begin{proof}  Suppose that $0 \to M' \to M \to M'' \to 0$ is an exact sequence in ${\cal D}_r$.  Express $M''$ as a direct limit $\varinjlim_{\lambda} \, A_\lambda$ of finite-dimensional modules.  Each $A_\lambda$ is in ${\bf p}_r$, so ${\rm Ext}^1(A_\lambda,M')=0$ and hence each pullback sequence below is split.

$\xymatrix{0 \ar[r] & M' \ar[r] & M \ar[r] & M'' \ar[r] & 0 \\
0 \ar[r] & M' \ar[r] \ar[u] & X_\lambda \ar[r] \ar[u] & A_\lambda \ar[r] \ar[u] & 0
}$

These fit together ($X_\lambda$ is just the full inverse image of $A_\lambda$ in $M$) into a directed system of split exact sequences, with limit the original exact sequence which is, therefore (see \cite[2.1.4]{PreNBK}), pure-exact.
\end{proof}

\begin{lemma}\label{iminD} \marginpar{iminD} If $M' \xrightarrow{f} M$ with $M$, $M'$ both in ${\cal D}_r$ then ${\rm im}(f)\in {\cal D}_r$.
\end{lemma}
\begin{proof}  Since $M'' ={\rm im}(f)$ embeds in $M$, $({\bf q}_r,M'')=0$.  Since $M''$ is an epimorphic image of $M'$, $(M'',{\bf p}_r)=0$, so $M'' \in {\cal B}_r \, \cap {\cal C}_r ={\cal D}_r$, as required.
\end{proof}

The category ${\cal D}_r$ is not, however, closed in ${\rm Mod}\mbox{-}R$ under kernels and cokernels.  Indeed, as we have seen in \ref{tiltapprox}, for any finite-dimensional module $A$ of slope $<r$ there is an exact sequence $0 \to A \to T_0 \xrightarrow{f} T_1 \to 0$ where $T_0, T_1 \in {\rm Add}(T) \subseteq {\cal D}_r$.  Dually, any finite-dimensional module of slope $>r$ is the cokernel of a morphism $g:C_0 \to C_1$ in ${\cal D}_r$ with $C_0, C_1\in {\rm Prod}(C)$.  In Section \ref{secnonpure} will see more precisely what are the morphisms with kernel not in ${\cal D}_r$.

\begin{lemma}\label{puregencogen} \marginpar{puregencogen} (\cite[proof of 6.4]{AngKusAlg}) For every $D\in {\cal D}_r$ there is an exact sequence 
$$0 \to T_1 \xrightarrow{f} T_0 \xrightarrow{p} D \to 0$$

\noindent with $T_0, T_1 \in {\rm Add}(T)$, $p$ a pure epimorphism and the inclusion ${\rm im}(f) \to T_0$ a pure monomorphism; and there is an exact sequence

$$0 \to D \xrightarrow{i} C_0 \xrightarrow{g} C_1 \to 0$$

\noindent with $C_0, C_1 \in {\rm Prod}(C)$, $i$ a pure monomorphism and $C_0 \to C_0/D$ a pure epimorphism.
\end{lemma}
\begin{proof}  Since ${\cal D}_r ={\rm Pres}(T)$, there is an exact sequence $T_1 \xrightarrow{f} T_0 \to D \to 0$ with $T_0, T_1 \in {\rm Add}(T)$.  By \ref{iminD}, ${\rm im}(f) \in {\cal D}$ so we have an exact sequence $0 \to {\rm im}(f) \to T_0 \to D \to 0$ in ${\cal D}_r$ which, by \ref{espes}, is pure-exact.

For the second statement, since ${\cal D}_r = {\rm Copres}(C)$, we have an exact sequence $0 \to D \xrightarrow{i} C_0 \to C_1$ with $C_0, C_1 \in {\rm Copres}(C)$.  Since $C_0/{\rm im}(i) \in {\cal D}$, we have, by \ref{espes}, that the exact sequence $0 \to D \to C_0 \to C_0/D \to 0$ is pure-exact.
\end{proof}

Recall, \ref{siltstrDat}, that every tilting module $T$ for ${\cal D}_r$ is strictly ${\cal D}_r$-atomic and some finite power of it is a ${\cal D}_r$-preenvelope for $R$.  

\begin{prop}\label{fpinT} \marginpar{fpinT} Let $A\in {\rm mod}\mbox{-}R$.  Then there is a morphism $A \to T$ for some tilting module $T$ for ${\cal D}_r$ such that this is a strictly ${\cal D}_r$-atomic, ${\cal D}_r$-preenvelope for $A$.
\end{prop}
\begin{proof}  Choose, by \ref{Mak}, some strictly atomic ${\cal D}_r$-preenvelope $f:A \to D_A$ for $A$.  There is, by \ref{puregencogen}, a pure epimorphism $p:T \to D_A$ for some tilting module $T$ for ${\cal D}_r$.  Suppose that $\overline{a}$ is a generating tuple for $A$, and let $\phi$ be such that ${\rm pp}^A(\overline{a}) = \langle \phi \rangle$.  Since $p$ is a pure epimorphism there is a tuple $\overline{b} \in \phi(T)$ with $p\overline{b} = f\overline{a}$ hence, by \ref{Denvpp}, with ${\rm pp}^T(\overline{b}) = {\rm pp}^{D_A}(f\overline{a})$ being ${\cal D}_r$-generated by $\phi$.  Therefore, by \ref{siltstrDat}, the morphism $A \to T$ given by $\overline{a} \mapsto \overline{b}$ is a strictly ${\cal D}_r$-atomic, ${\cal D}_r$-preenvelope for $A$.
\end{proof}

\begin{cor}\label{strDrat} \marginpar{strDrat}  Let $T$ be a tilting module for ${\cal D}_r$.  Then every strictly ${\cal D}_r$-atomic module in ${\cal D}_r$ is in ${\rm Add}(T)$.
\end{cor}
\begin{proof}  Suppose that $D \in {\cal D}_r$ is strictly ${\cal D}_r$-atomic.  Let $\overline{a}$ be a tuple from $D$, so ${\rm pp}^D(\overline{a})$  is generated, modulo the theory of ${\cal D}_r$ by a pp formula, $\phi$, say.  Let $(C_\phi,\overline{c}_\phi)$ be a free realisation of $\phi$.  By \ref{fpinT} there is a ${\cal D}_r$-preenvelope $f_\phi: C_\phi \to T_\phi$ with $T$ in ${\rm Add}(T)$.  By assumption, there is a morphism $g_{\overline{a}}: D\to T_\phi$ taking $\overline{a}$ to $f_\phi \overline{c}_\phi$.  Take the direct sum of all these morphisms $g_{\overline{a}}$ as $\overline{a}$ ranges over finite tuples in $D$.  Then this morphism is pp-type-preserving, hence a pure embedding.

Since ${\cal D}_r \subseteq {\rm Gen}(T) = T^{\perp_1}$, ${\rm Add}(T) \subseteq {^{\perp_1}}{\cal D}_r$.  By \cite[p.~846, Rmk.~1]{AngKusAlg} every module in ${\cal D}$ has injective dimension $\leq 1$ and hence ${^{\perp_1}}{\cal D}_r$ is closed under submodules, so ${\rm Ext}^1(D,{\cal D}_r)=0$.  But, \ref{puregencogen}, there is an exact sequence $0 \to T_1 \to T_0 \to D \to 0$ with $T_0, T_1 \in {\rm Add}(T)$.  So $D$ is a direct summand of $T_0$ and hence is in ${\rm Add}(T)$ as claimed.
\end{proof}

\subsection{Non-pure morphisms in ${\cal D}_r$}\label{secnonpure} \marginpar{secnonpure}

The next result and its extension that follows in some sense explain the non-pure surjections in ${\cal D}_r$.  First, note that, if $A$ is a finitely presented module and if $\overline{a}$ is a finite generating tuple of $A$, with $\theta_{\overline{a}}$ the conjunction of finitely many relations on $\overline{a}$ which generate all the $R$-linear relations on $\overline{a}$, then $\theta_{\overline{a}}$ generates the pp-type, ${\rm pp}^A(\overline{a})$, of $\overline{a}$ in $A$.

\begin{prop}\label{fgker} \marginpar{fgker}  Suppose that $A=\overline{a}R$ is a finitely generated submodule of $D\in {\cal D}_r$ and let $\theta_{\overline{a}}$ be a quantifier-free formula generating ${\rm pp}^A(\overline{a})$.  Then $D/A\in {\cal D}_r$ iff ${\rm pp}^D(\overline{a}) = \langle \theta_{\overline{a}} \rangle_{{\cal D}_r}$. 
\end{prop}
\begin{proof}  ($\Rightarrow$)  Suppose that $\phi \in {\rm pp}^D(\overline{a})$, that is $\phi$ is pp and $\overline{a} \in \phi(D)$.  If $(C_\phi, \overline{c})$ is a free realisation of $\phi$ then there is a morphism $C_\phi \to D$ taking $\overline{c}$ to $\overline{a}$, so we may assume that $C_\phi \in {\bf p}_r$.  By \ref{localpp} there is $\phi' \geq \phi$ and $\epsilon >0$ and a free realisation $(C_{\phi'}, \overline{c}')$ of $\phi'$ such that $C_{\phi'} \in {\bf p}_r$, $C_{\phi'}/\langle \overline{c}' \rangle \in {\bf q}_r$ and $\phi' \equiv \phi$ on $(r-\epsilon, r+\epsilon)$.  In particular $\phi' \equiv_{{\cal D}_r} \phi$.

Also, since there will therefore be a morphism $f:C_{\phi'} \to D$ with $\overline{c}' \mapsto \overline{a}$, there is an induced morphism $C_{\phi'}/\langle \overline{c} \rangle \to D/A$.  We are assuming that $D/A$ has slope $r$, so this must be the zero map and hence ${\rm im}(f) =A$.  Thus we have a morphism $C_{\phi'} \to A$ with $\overline{c}' \mapsto \overline{a}$ and we deduce that $\overline{a} \in \phi'(A)$.  Since $\overline{a} \in A$ freely realises $\theta_{\overline{a}}$, we deduce that $\phi' \geq \theta_{\overline{a}}$.

So, since $\phi' \equiv \phi$ on ${\cal D}_r$ (in fact, on a neighbourhood of $r$), we have $\phi \geq_{{\cal D}_r} \theta_{\overline{a}}$ and hence ${\rm pp}^D(\overline{a}) = \langle \theta_{\overline{a}} \rangle_{{\cal D}_r}$.

($\Leftarrow$)  For the converse, we have by \ref{tiltapprox} that there is an exact sequence $0 \to A \xrightarrow{i} L_0 \xrightarrow{p} L_1 \to 0$ with $L_0$ a ${\cal D}_r$-preenvelope of $A$ and $L_1 \in {\cal D}_r$.  So there is $f:L_0 \to D$ with $A\xrightarrow{i} L_0 \xrightarrow{f} D$ equal to the inclusion $A \leq D$.  By assumption and \ref{Denvpp} we have ${\rm pp}^D(\overline{a}) = {\rm pp}^{L_0}(i\overline{a})$.  We use this to show that if $\phi/\psi$ is a pp-pair closed on ${\cal D}_r$, then $\phi/\psi$ is closed on $D/A$, and hence $D/A \in {\cal D}_r$.

So suppose that $D/A \models \phi(\overline{c}')$ where $\phi(\overline{x})$ is $\exists \overline{y} \, \theta(\overline{x}, \overline{y})$ where $\theta$ is 
$$\bigwedge_j \, \sum_i x_ir_{ij} + \sum_k y_ks_{kj} =0,$$ 
say $D/A \models \theta(\overline{c}', \overline{d}')$ for some $\overline{d}'$ in $D/A$.  So we have 
$$\bigwedge_j \, \sum_i c'_ir_{ij} + \sum_k d'_ks_{kj} =0.$$  
Choose inverse images $c_i$ of $c_i'$ and $d_j$ of $d_j'$ in $D$ and also choose $a_j \in iA$ such that $$\bigwedge_j \, \sum_i c_ir_{ij} + \sum_k d_ks_{kj} = fa_j.$$  
Therefore $$D \models \exists \overline{x}\,  \overline{y} \, \bigwedge_j \, \sum_i x_ir_{ij} + \sum_k y_ks_{kj} =fa_j$$ 
and so the formula $$\exists \overline{x}\, \overline{y} \, \bigwedge_j \, \sum_i x_ir_{ij} + \sum_k y_ks_{kj} =z_j$$ 
is in ${\rm pp}^D(f\overline{a}) = {\rm pp}^{L_0}(i\overline{a})$.  Therefore $$L_0 \models \bigwedge_j \, \sum_i m_ir_{ij} + \sum_k n_ks_{kj} =a_j$$ 
for some $m_i, n_j \in L_0$ and hence $$L_1 \models \phi(p\overline{m}).$$  
Note that $$D \models \bigwedge_j \, \sum_i fm_ir_{ij} + \sum_k fn_ks_{kj} =fa_j$$ follows and hence $$D \models \bigwedge_j \, \sum_i (c_i -fm_i)r_{ij} + \sum_k (d_k -fn_k)s_{kj} =0,$$ 
that is, $D\models \theta(\overline{c}-f\overline{m}, \overline{d}-f\overline{n})$ and hence $D \models \phi(\overline{c}-f\overline{m})$.  We are assuming $\phi/\psi$ to be closed on $D$ and therefore $D\models \psi(\overline{c} -f\overline{m})$ and so $D/A \models \psi(\overline{c}' -\pi f\overline{m})$ where $\pi:D \to D/A$ is the projection.

We know that $\phi/\psi$ is also closed on $L_1$ where $\psi(\overline{x})$ is, say, $\exists \overline{u} \, \theta'(\overline{x}, \overline{u})$ and $\theta'$ is $\bigwedge_t \, \sum_i x_ir'_{it} + \sum_l u_ls'_{lt} =0$.  Therefore there are $e'_l \in L_1$ such that
$$L_1 \models \bigwedge_t \, \sum_i pm_ir'_{it} + \sum_l e'_ls'_{lt} =0.$$  
So there are $e_l$ in $L_0$ with $pe_l =e'_l$ and there are $a'_t \in iA$ such that 
$$L_0 \models \bigwedge_t \, \sum_i m_ir'_{it} + \sum_l e_ls'_{lt} =a'_t$$ and hence such that 
$$D \models \bigwedge_t \, \sum_i fm_ir'_{it} + \sum_l fe_ls'_{lt} =fa'_t.$$  
We deduce that 
$$D/A \models \bigwedge_t \, \sum_i \pi fm_ir'_{it} + \sum_l \pi fe_ls'_{lt} =0,$$ that is $D/A \models \theta'(\pi f\overline{m}, \pi f \overline{u})$, hence $D/A \models \psi(\pi f\overline{m})$.  Combined with the conclusion of the previous paragraph, this gives $D/A\models \psi(\overline{c}')$, as required.
\end{proof}

That is, if a finitely generated submodule $A$ of $D\in {\cal D}_r$ has its pp-type\footnote{equivalently, the pp-type of any generating tuple.} in $D$ being the minimal possible - that is, ${\cal D}_r$-generated by its isomorphism type - then $D/A \in {\cal D}_r$ (and conversely - that is stated formally as \ref{minppfg} below).  Thus we have a source of morphisms in ${\cal D}_r$ with kernel not in ${\cal D}_r$.

We have the following extension of \ref{fgker} which identifies the kernels of morphisms in ${\cal D}_r$ as the definably closed submodules of modules in ${\cal D}_r$ (note that a pure submodule is definably closed).

\begin{theorem}\label{kerinDr} \marginpar{kerinDr}  Suppose that $K\subseteq D \in {\cal D}_r$.  Then $D/K \in {\cal D}_r$ iff $K$ is definably closed in $D$.
\end{theorem}
\begin{proof}  Set $\pi: D \to D/K$ to be the projection map.

($\Rightarrow$)  We have seen this argument already:  suppose that $b\in {\rm dcl}^D(K)$; say $D\models \rho(\overline{a}, b)$ with $\overline{a}$ from $K$, $\rho$ pp and $\rho(\overline{0},D) =0$, hence also $\rho(\overline{0},D/K)=0$ by assumption and \ref{nodefsub}.  Then $D/K\models \rho(\overline{0}, \pi b)$, so $\pi b=0$ and $b\in K$, as required.

($\Leftarrow$)  The argument is a modification of that for \ref{fgker}. 

Suppose that the pp-pair $\phi/\psi$ is closed on ${\cal D}_r$; we show that $\phi/\psi$ is closed on $D/K$, which will be enough.

So suppose that $D/K \models \phi(\overline{c}')$ where $\phi(\overline{x})$ is $\exists \overline{y} \, \theta(\overline{x}, \overline{y})$ with $\theta$ being 
$$\bigwedge_j \, \sum_i x_ir_{ij} + \sum_k y_ks_{kj} =0,$$ 
say $D/K \models \theta(\overline{c}', \overline{d}')$ for some $\overline{d}'$ in $D/K$.  So we have 
$$\bigwedge_j \, \sum_i c'_ir_{ij} + \sum_k d'_ks_{kj} =0.$$
Choose inverse images $c_i$ of $c_i'$ and $d_j$ of $d_j'$ in $D$ and also choose $a_j \in K$ such that $$D\models \bigwedge_j \, \sum_i c_ir_{ij} + \sum_k d_ks_{kj} = a_j.$$  Therefore $$D \models \exists \overline{x}\,  \overline{y} \, \bigwedge_j \, \sum_i x_ir_{ij} + \sum_k y_ks_{kj} =a_j$$ and so the formula $\tau(\overline{v})$ which is $$\exists \overline{x}\, \overline{y} \, \bigwedge_j \, \sum_i x_ir_{ij} + \sum_k y_ks_{kj} =v_j$$ is in ${\rm pp}^D(\overline{a})$.  By \ref{ppdcl} there is a pp formula $\exists \overline{z} \, \theta_0(\overline{z}, \overline{v}) \in {\rm pp}^K(\overline{a})$ with $\theta_0$ quantifier-free such that $\exists \overline{z} \, \theta_0(\overline{z}, \overline{v}) \leq_{{\cal D}_r} \tau(\overline{v})$.  Say we have $\theta_0(\overline{\kappa}, \overline{a})$ with $\kappa$ from $K$.  Set $K_0 = \langle \overline{a}, \overline{\kappa} \rangle$ to be the module generated by the entries of these tuples.  Note that $K_0 \models \theta_0(\overline{\kappa}, \overline{a})$.

There is, since $K_0$ is finitely generated and is a submodule of $D \in {\cal D}_r$, an exact sequence $0 \to K_0' \to L_0 \to L_1 \to 0$ with $K_0'$ a copy of $K_0$, $L_0$ a ${\cal D}_r$-preenvelope of $K_0'$ and $L_1 \in {\cal D}_r$.  So there is $f:L_0 \to D$ which restricts to an isomorphism on $K_0' \simeq K_0$.

Since $K_0 \models \theta_0(\overline{\kappa}, \overline{a})$, we have $K_0' \models \exists \overline{z} \, \theta_0(\overline{z}, \overline{a}_0)$, where we write $\overline{a}_0$ for the copy of $\overline{a}$ in $K_0'$.  Therefore $\exists \overline{z} \, \theta_0(\overline{z}, \overline{v}) \in {\rm pp}^{L_0}(\overline{a}_0)$ (we identify $K_0'$ with its image in $L_0$) and so, by choice of $\theta_0$, we have $L_0 \models \tau(\overline{a}_0)$.  Say 
$$L_0 \models \bigwedge_j \, \sum_i m_ir_{ij} + \sum_k n_ks_{kj} =a_{0j}$$ 
for some $m_i, n_j \in L_0$ and hence 
$$L_1 \models \phi(p\overline{m}).$$  
Noting that $$D \models \bigwedge_j \, \sum_i fm_ir_{ij} + \sum_k fn_ks_{kj} =fa_{0j}, = a_j$$ 
we proceed from here exactly as in the proof of \ref{fgker}.
\end{proof}

This lets us say precisely how the morphisms in ${\cal D}_r$ with kernel not in ${\cal D}_r$ are associated with minimal pp-types in ${\cal D}_r$.  

\begin{cor}\label{LHmors} \marginpar{LHmors}  Suppose $f:D\to D'$ with $D$, $D'$ in ${\cal D}_r$ and $K = {\rm ker}(f)$ supported on $(-\infty, r-\eta)$ for some $\eta >0$, in particular, $K \notin {\cal D}_r$.  Then for every finite tuple $\overline{a}$ from $K$, ${\rm pp}^D(\overline{a}) = \langle {\rm pp}^K(\overline{a}) \rangle_{{\cal D}_r}$. That is, ${\rm pp}^D(K)$ is the minimal pp-type realised in ${\cal D}_r$ extending the isomorphism type of $K$.
\end{cor}
\begin{proof}  This follows directly from \ref{kerinDr} and \ref{ppdcl} (the latter is stated for finitely generated modules but the general case is an immediate consequence of that).  But the proof direct from \ref{localpp} is quick, so we also give this.

Take any tuple $\overline{a}$ from $K$ and suppose that $\phi$ is a pp formula such that $D \models \phi(\overline{a})$.  By \ref{localpp} there is a pp formula $\phi'$ equivalent to $\phi$ at (and near) $r$ and with a free realisation $(C,\overline{c})$ such that $C/\langle \overline{c} \rangle \in {\bf q}_r$.

Then we have a morphism $g:C \to D$ with $g\overline{c} = \overline{a}$ and hence an induced morphism $C/\langle \overline{c} \rangle \to D'$ which, since the slope of $C/\langle \overline{c} \rangle$ is greater than $r$, must be $0$.  Hence $gC \leq K$.  But then $K\models \phi'(\overline{a})$ and so, since $\phi$ is equivalent to $\phi'$ near $r$, $\phi$ is in the ${\cal D}_r$-closure of ${\rm pp}^K(\overline{a})$, as required.
\end{proof}

\begin{cor}\label{minppfg} \marginpar{minppfg} Suppose that $0 \to A \to D \to D'$ is an exact sequence with $D, D' \in {\cal D}_r$ and $A$ finite-dimensional, generated by the $n$-tuple $\overline{a}$.  Then ${\rm pp}^D(\overline{a})$ is generated, modulo ${\cal D}_r$, by any quantifier-free formula which generates the defining linear relations on $\overline{a}$.  In particular it is the minimal pp-type of any tuple $\overline{c}$ from a module in ${\cal D}_r$ with the same isomorphism type as $\overline{a}$ (that is, with $(\overline{c}R,\overline{c}) \simeq (\overline{a}R, \overline{a})$ as $n$-pointed modules).
\end{cor}

That follows by \ref{iminD} and \ref{fgker}.

We look at the following case more closely.  Note that by \ref{tiltapprox}, every $A\in {\rm mod}\mbox{-}R$ supported on $(-\infty,r)$ has a ${\cal D}_r$-preenvelope which is a monomorphism.

\begin{prop} Suppose that $A\in {\rm mod}\mbox{-}R$ is supported on $(-\infty,r)$ and let $A \to D_A$ be a ${\cal D}_r$-atomic ${\cal D}_r$-preenvelope.  Then $D_A/A \in {\cal D}_r$ and $D_A/A$ is ${\cal D}_r$-atomic.  If $D_A$ is strictly ${\cal D}_r$-atomic, so is $D_A/A$.
\end{prop}
\begin{proof}  The fact that $D_A/A \in {\cal D}_r$ is by \ref{Denvpp} and \ref{fgker}.  Let $\overline{b}$ be from $D_A$ and choose $\phi(\overline{x}, \overline{y})$ which ${\cal D}_r$-generates ${\rm pp}^{D_A}(\overline{a}, \overline{b})$ where $\overline{a}$ is a chosen finite generating tuple for $A$.  We claim that $\phi(\overline{0}, \overline{y})$ ${\cal D}_r$-generates ${\rm pp}^{D_A/A}(\pi\overline{b})$, where $\pi:D_A \to D_A/A$ is the quotient map.

Certainly that formula is in ${\rm pp}^{D_A/A}(\pi \overline{b})$, so suppose that $D_A/A \models \psi(\pi \overline{b})$, say $\psi$ is $\exists \overline{z} \, \bigwedge_j \, \sum y_ir_{ij} + \sum _k z_ks_{kj} =0$, so $D_A \models \bigwedge_j \, \sum b_ir_{ij} + \sum c_k s_{kj} = a_j'$ for some $c_k\in D_A$ and $a_j' \in A$.  Then $\exists \overline{z} \, \bigwedge_j \, \sum_i y_ir_{ij} + \sum_k z_ks_{kj} = x_j'$ is a consequence (modulo ${\cal D}_r$) of $\phi(\overline{x}, \overline{y})$ where $x_j'$ is being used for the linear combination of the variables $\overline{x}$ that corresponds to $a_j'$ written as a specific linear combination of the entries of $\overline{a}$.  That is, $\phi(\overline{x}, \overline{y}) \leq_{{\cal D}_r} \exists \overline{z} \, \bigwedge_j \, \sum_i y_ir_{ij} + \sum_k z_ks_{kj} = x_j'$, so $\phi(\overline{0}, \overline{y}) \leq_{{\cal D}_r} \exists \overline{z} \, \bigwedge_j \, \sum_i y_ir_{ij} + \sum_k z_ks_{kj} = 0$, that is $\phi(\overline{0}, \overline{y}) \leq_{{\cal D}_r} \psi(\overline{y})$, as claimed.

So every pp-type realised in $D_A/A$ is finitely generated modulo the theory of ${\cal D}_r$; that is, $D_A/A$ is ${\cal D}_r$-atomic.   Suppose that $D_A$ is strictly ${\cal D}_r$-atomic and, continuing the notation as above, take a finite tuple $\pi \overline{b}$ from $D_A/A$ and a ${\cal D}_r$-generator $\phi(\overline{0}, \overline{y})$ for the pp-type of $\pi \overline{b}$ in $D_A/A$ constructed as above.  Suppose that $D\in {\cal D}_r$ and $D \models \phi(\overline{0}, \overline{d})$.  Then, by choice of $\phi$ and since $D_A$ is strictly ${\cal D}_r$-atomic, there is a morphism $D_A \to D$ with $\overline{a} \mapsto \overline{0}$ and $\overline{b} \mapsto \overline{d}$, so this morphism factors through $\pi$, giving a morphism $D_A/A \to D$ with $\pi \overline{b} \mapsto \overline{d}$, as required.
\end{proof}

Here's a little more about morphisms of ${\cal D}_r$ with kernel $R$.

\begin{prop}  Let $T$ be a tilting module in ${\cal D}_r$ and ${C}$ a cotilting module in ${\cal D}_r$.  If $0 \to R  \xrightarrow{i}  T_0 \to T_1 \to 0$ is a ${\cal D}_r$-preenveloping sequence with $T_0, T_1 \in {\rm Add}(T)$, then there is an exact sequence $0 \to R \to H(T_0) \to H(T_0)/R \to 0$ with $H(T_0), H(T_0)/R$ in ${\rm Prod}(C)$, and the induced inclusion of $H(T_1)$ in $H(T_0)/R$ is split.
\end{prop}
\begin{proof}  Consider a pure-injective hull $i': T_0 \to H(T_0) \in {\cal D}_r$.  Set $a=i(1)$; then ${\rm pp}^{H(T_0)}(i'a) = {\rm pp}^{T_0}(a) = \langle x=x\rangle_{{\cal D}_r}$ is the  generic pp-type (that is, the smallest pp-type, being generated by ``$x=x$") in ${\cal D}_r$; write this as $p_0$.  We also have (\cite[proof of 6.4]{AngKusAlg}) an exact sequence $0 \to R \xrightarrow{j} C_0 \to C_1 \to 0$ with $C_0, C_1 \in {\rm Prod}(C)$ where $C$ is a cotilting module for ${\cal D}_r$ as at the beginning of Section \ref{secirrat}; in particular these are pure-injective (every cosilting module is pure-injective).  Since $j$ is the kernel of a morphism in ${\cal D}_r$, we have by \ref{kerinDr} that $jR$ is definably closed in $C_0$.  Then \ref{ppdcl} implies that ${\rm pp}^{C_0}(j(1)) = p_0$ and so $H(T_0)$ is a direct summand of $C_0$.  Therefore we can replace this exact sequence with $0 \to R \to H(T_0) \to H(T_0)/R \to 0$, deducing in particular, that $H(T_0)/R$, being a direct summand of $C_1$, is pure-injective.
Next consider the diagram.
$\xymatrix{
0 \ar[r] & R \ar[r] \ar@{=}[d] & T_0 \ar[r] \ar[d]_{i'} & T_1 \ar[r] \ar[d]^f & 0 \\
0 \ar[r] & R \ar[r] & H(T_0) \ar[r] & H(T_0)/R \ar[r] & 0
}$
Since $i'$ is a pure embedding, so is its pushout $f$, so $H(T_1)$ is a direct summand of $H(T_0)/R$.
\end{proof}

The modules $H(T_0)$ and $H(T_1)$, although in ${\rm Prod}(C)$, certainly are not cotilting modules since, according to next result, they have no neg-isolated direct summands.

\begin{prop}  If $T$ is a strictly ${\cal D}_r$-atomic module then the pure-injective hull $H(T)$ of $T$ has no neg-isolated direct summand.
\end{prop}
\begin{proof}  We need rather more model theory/functor category theory for this.  We use the embedding $M \mapsto M\otimes -$ of ${\rm Mod}\mbox{-}R$ into the functor category $(R\mbox{-}{\rm mod}, {\bf Ab})$ followed by Gabriel localisation at the torsion theory which is generated by the finitely presented functors which are $0$ on the dual definable category ${\cal D}_r^{\rm d}$.  In $(R\mbox{-}{\rm mod}, {\bf Ab})$, $H(T)\otimes -$ is the injective hull of $T\otimes -$ and is torsionfree for that torsion theory.  Working in the localised category (see \cite[\S 12.5, 12.5.6 especially]{PreNBK}), if $H(T)$ has a neg-isolated direct summand, say the hull $H(p)$ of a pp-type $p$ neg-isolated by a pp formula $\psi$, then $H(T)\otimes -$ has a simple subobject, namely, the localisation of the functor $F_{D\psi/Dp}$, and hence so does its essential subobject $T\otimes -$.  Therefore, $T$ realises a neg-isolated type - for we have a nonzero morphism $F_{D\psi/Dp} \to T\otimes -$ and so, by \cite[12.2.4]{PreNBK}, there is\footnote{We obtain the element $a$ as follows.  Since $F_{D\psi}$ is a subfunctor of the forgetful functor $(_RR,-) \simeq (R\otimes_R-)$, we have an inclusion $F_{D\psi/Dp} \to (R\otimes -)/F_{Dp}$ with finitely presented cokernel $(R\otimes -)/F_{D\psi}$.  Since $(T\otimes -)$ is absolutely pure, the morphism $F_{D\psi/Dp} \to T\otimes -$ therefore extends to a morphism $(R\otimes -)/F_{Dp} \to (T\otimes -)$.  By composition with $(R\otimes -) \to (R\otimes -)/F_{Dp}$ we obtain a morphism $(R\otimes -) \to (T\otimes -)$, which must be induced by a morphism $R_R \to T$, that is, by an element $a\in T$. The pp-type of $a$ in $T$ is $p$ by \cite[12.2.5]{PreNBK} and the fact that $F_{Dp}$ is, because $(R\otimes -)/F_{Dp}$ is uniform by \cite[12.2.3]{PreNBK} with a simple subobject, the kernel of $(a\otimes -):(R\otimes -) \to (T\otimes -)$.} $a\in T$ with ${\rm pp}^T(a) =p$.  But every pp-type realised in $T$ is finitely generated, so $p = \langle \phi \rangle_{{\cal D}_r}$ for some pp formula $\phi$.  But then $\phi /\psi$ is a minimal pair in the ordering $\leq_{{\cal D}_r}$, meaning there is no point in the ordering strictly between them.  But that contradicts \cite[6.1, 7.3]{HarPre}, so we have a contradiction as required. (In terms of the functor category, the localisation of the object $F_{D\psi}/F_{D\phi}$ is equal to the localisation of $F_{D\psi}/F_{Dp}$, which shows that this simple object is finitely presented in the localised functor category, contradicting the result in \cite{HarPre}.)
\end{proof}

\noindent {\bf Question:}  Are there any nonzero objects in ${\cal D}_r$ which are finitely presented in ${\cal D}_r$?

\vspace{4pt}

We can say this much:

\begin{prop}  If $D\in {\cal D}_r$ is finitely presented in ${\cal D}_r$, then $D$ is ${\cal D}_r$-atomic.  Indeed, every finite tuple of $D$ can be extended to a finite tuple whose pp-type is ${\cal D}_r$-generated by its quantifier-free type (cf.~\cite[3.13]{RotHab}).
\end{prop}
\begin{proof}  Write $D =\varinjlim \, A$ as the direct limit of its finitely generated submodules.  For each finitely generated submodule $A$ of $D$, choose a ${\cal D}_r$-atomic, ${\cal D}_r$-precover $A \to D_A$ of $A$.  By \cite[3.3]{LackTend} these may be chosen in a functorial way, so that, corresponding to an inclusion $A \leq B$ of finitely generated submodules of $D$, we have a morphism $g_{AB}:D_A \to D_B$ and these morphisms give a directed system, with $\varinjlim D_A = D_1$ say.  Since $D = \varinjlim A$ there is an induced morphism $f:D \to D_1$ (indeed, this also is functorial, as stated in \cite[3.3]{LackTend}).

Since $D$ is finitely presented in ${\cal D}_r$, there is $A \leq D$ finitely generated and a morphism $h:D \to D_A$ such that $f=g_{A\infty}h$ where $g_{A\infty}:D_A \to D_1$ is the limit map.  Let $\overline{b}$ be any tuple from $D$ and, without loss of generality, assume that it contains a generating tuple for $A$.  Set $B$ to be the submodule of $D$ generated by $\overline{b}$.  Then we have ${\rm pp}^B(\overline{b}) \leq {\rm pp}^D(\overline{b}) \leq {\rm pp}^{D_A}(\overline{hb}) \leq {\rm pp}^{D_B}(\overline{g_{AB}b})$.  The last pp-type is ${\cal D}_r$-finitely generated, being equivalent, modulo the theory of ${\cal D}_r$, to the first pp-type (\ref{Denvpp}) and hence is generated by any pp formula which generates the first pp-type.  Hence ${\rm pp}^D(\overline{b})$ is generated, modulo ${\cal D}_r$, by (any quantifier-free pp formula which generates) ${\rm pp}^B(\overline{b})$. 
\end{proof}

\section{Background from Model Theory}\label{secmodth} \marginpar{secmodth}

This consists of brief explanations; for more information and detail there are various references, including the comprehensive \cite{PreBk} and \cite{PreNBK} and the introductions to many other works, such as \cite{RotHab}, \cite{RotML2}.

\paragraph{Pp formulas}  
A {\bf pp formula} $\phi$ is (one which is equivalent to) an existentially quantified system of $R$-linear equations, that is, has the form
$$\exists \overline{y} \, \bigwedge_{j=1}^m \, \sum_{i=1}^n x_ir_{ij} +\sum_{k=1}^t y_ks_{kj} =0.$$
Here the $r_{ij}$ and $s_{kj}$ are elements of $R$ (precisely, function symbols standing for multiplication by those elements) and $\bigwedge$ is used for repeated conjunction $\wedge$, where the conjunction symbol $\wedge$ means ``and"; so this is a system of $m$ $R$-linear equations.  The variables $\overline{x} = (x_1, \dots, x_n)$ are the {\bf free} variables of $\phi$ (they are `free' to be substituted with values from some module) and the $y_k$ are the existentially quantified variables.  We may display the free variables of $\phi$, writing $\phi(\overline{x})$ or $\phi(x_1,\dots, x_n)$.

A {\bf quantifier-free} pp formula is a formula (equivalent to one) with no existential quantifiers.  For instance $\bigwedge_{j=1}^m \, \sum_{i=1}^n x_ir_{ij} +\sum_{k=1}^t y_ks_{kj} =0$ is a quantifier-free formula, with free variables the $x_i$ and the $y_k$.

\paragraph{Solution sets of pp formulas}  
If $\phi =\phi(x_1,\dots, x_n)$ is the pp formula above then, in any module $M$, we define its solution set:  
$$\phi(M) = \{ (a_1,\dots, a_n)\in M^n: \exists b_1,\dots, b_t\in M \text{ such that } \sum_{i=1}^n a_ir_{ij} +\sum_{k=1}^t b_ks_{kj} =0, \, \forall j\}.$$
This is a projection, to the first $n$ coordinates in $M^{n+t}$, of the solution set of the quantifier-free formula $\bigwedge_{j=1}^m \, \sum_{i=1}^n x_ir_{ij} +\sum_{k=1}^t y_ks_{kj} =0$.  Since the solution set to the latter is a subgroup of $M^{n+t}$, its projection $\phi(M)$ is a subgroup of $M^n$.  (In fact, it is easy to see that both are ${\rm End}(M)$-submodules under the diagonal action of that ring on powers of $M$.)  We say that $\phi(M)$ is a subgroup of $M^n$ pp-definable in $M$ or, more briefly though, if $n>1$, less accurately, a {\bf pp-definable subgroup} of $M$.

If $\overline{a} \in \phi(M)$ then we write $M \models \phi(\overline{a})$ - this is the more usual notation in model theory and is read as ``$\overline{a}$ satisfies $\phi$ in $M$".

Since pp formulas define subgroups, we have that $M\models \phi(\overline{a})$ and $M \models \phi(\overline{b})$ together imply $M \models \phi(\overline{a} - \overline{b})$.

\paragraph{The (pre-)ordering on, and equivalence of, pp formulas}
We write $\psi \leq \phi$ if, for every module $M$, $\psi(M) \leq \phi(M)$.  We will make this comparison only when $\psi$ and $\phi$ have the same free variables (so that $\psi(M)$ and $\phi(M)$ may be compared as subsets of the same power of $M$).  This is a preordering, and {\bf equivalence} of formulas means equivalence with respect to this.  More generally, we say that $\phi$ is {\bf equivalent} to $\psi$ {\bf in} $M$ if $\phi(M) =\psi(M)$, that is if $M\models \phi(\overline{a})$ iff $M \models \psi(\overline{a})$.  So two pp formulas are equivalent iff this holds for every $M$ (in fact, to test the ordering and equivalence it is enough to check just on finitely presented modules \cite[1.2.23]{PreNBK}).  So in practice we use ``$=$" not to mean that the formulas are identical (as strings of symbols) but rather to mean that they have the same solution sets.

\paragraph{Lattices of pp formulas}
For each $n$ the resulting ordered set of (equivalence classes of) pp formulas in (a specified list of) $n$ free variables is a modular lattice, written ${\rm pp}^n_R$, the point being that each of the intersection and sum of $\phi(M), \psi(M) \leq M^n$ is the solution set of a pp formula; those formulas are respectively written $\phi \wedge \psi$ and $\phi + \psi$ and are entirely independent of $M$.  Explicitly, $\phi \wedge \psi$ is the usual formal conjunction of formulas and $\phi + \psi$ is $\exists \overline{x}_1, \overline{x}_2\, (\overline{x} = \overline{x}_1 + \overline{x}_2 \, \wedge \, \phi(\overline{x}_1) \, \wedge \, \psi(\overline{x}_2))$.

So, for every module $M$, we have the evaluation map ${\rm pp}^n_R \to {\rm pp}^n(M)$ where the latter is the set, indeed modular lattice, of subgroups of $M^n$ pp-definable in $M$.  The kernel of this lattice homomorphism consists of the pairs $(\phi, \psi)$ such that $\phi(M) = \psi(M)$:  we say that such a pair is {\bf closed on} $M$.  Otherwise the pair is {\bf open on} $M$.  A {\bf pp-pair} is a pair of pp formulas which are comparable, $\phi \geq \psi$, in the ordering on ${\rm pp}^n_R$.

We write $\psi \leq_M \phi$ and $\psi =_M \phi$ for the (pre)ordering and equivalence of pp formulas when evaluated on $M$.

\paragraph{Definable subcategories}
Given any set $\Phi$ of pp-pairs, the corresponding {\bf definable subcategory} of ${\rm Mod}\mbox{-}R$ is the full subcategory on
$$\{ M \in {\rm Mod}\mbox{-}R: \phi(M) =\psi(M) \,\, \forall (\phi,\psi) \in \Phi\}.$$
Thus a definable subcategory is one with membership determined by closure of a certain set of pp-pairs.

The definable subcategories of ${\rm Mod}\mbox{-}R$ are characterised algebraically as being those closed under direct products, directed colimits and pure submodules (\cite[3.4.7]{PreNBK}).  They also are closed under pure epimorphisms and pure-injective hulls (\cite[3.4.8]{PreNBK}).  A {\bf definable category} is one which is equivalent to a definable subcategory of some module category ${\rm Mod}\mbox{-}R$ (we allow $R$ to be a ring with many objects, that is a skeletally small preadditive category).

If $M$ is a module then we denote by $\langle M \rangle$ the definable subcategory {\bf generated} by $M$ - the smallest definable subcategory (of the ambient module category) containing $M$:
$$ \langle M \rangle = \{ N\in {\rm Mod}\mbox{-}R: \phi(M) = \psi(M) \implies \phi(N) =\psi(N) \,\, \forall  \phi, \psi \text{ pp }\}.$$
That is, $\langle M \rangle$ consists of the class of modules $N$ such that every pp-pair closed on $M$ is closed on $N$.  Similar notation is used for the definable subcategory generated by a class of modules.
Every definable subcategory is generated by some (by no means unique) $M$.

\paragraph{The functor category of a definable category}
If ${\cal D}$ is a definable category, then the functors from ${\cal D}$ to the category, ${\bf Ab}$, of abelian groups which commute with direct products and directed colimits are precisely those given by pp-pairs:  those of the form $D \mapsto \phi(D)/\psi(D)$, for $\phi \geq \psi$ a pp-pair, see \cite[18.1.19]{PreNBK} (and the main result of \cite{Makk} specialises to something close to this).  This category is also equivalent to the localisation of the functor category $({\rm mod}\mbox{-}R, {\bf Ab})^{\rm fp}$ - the finitely presented functors on finitely presented modules - by the Serre subcategory consisting of those finitely presented functors which are $0$ on ${\cal D}$.  Indeed, the finitely presented functors being just the pp-pairs, these are exactly all the pp-pairs which are closed on, hence which together define, ${\cal D}$.  See \cite[12.3.19, 12.3.20]{PreNBK}.

\paragraph{Pp formulas relative to a definable subcategory}
If ${\cal D}$ is a definable subcategory, then we write $\psi \leq_{\cal D} \phi$ if $\psi(M) \leq \phi(M)$ for every $M \in {\cal D}$, and $\psi =_{\cal D} \phi$ if $\psi(M) = \phi(M)$ for every $M \in {\cal D}$.  If ${\cal D} = \langle M \rangle$, then these are the same as $\leq_M$ and $=_M$.

The relation $ =_{\cal D}$ of ${\cal D}$-equivalence between pp formulas in a given set of, say $n$, free variables is a congruence on the lattice ${\rm pp}^n_R$ of pp formulas in those $n$ free variables and so there is induced a surjective lattice homomorphism to the lattice ${\rm pp}^n_{\cal D}$ of equivalent-on-${\cal D}$ classes of pp formulas (which can be identified with ${\rm pp}^n(M)$ if $\langle M \rangle = {\cal D}$).

\paragraph{Elementary duality of pp formulas}
If $\phi(\overline{x})$ is a pp formula for right $R$-modules then there is an ({\bf elementary}) {\bf dual} pp formula $D\phi(\overline{x})$ for left $R$-modules (in the same number of free variables\footnote{Free variables are just place-holders so it doesn't matter whether or not we use the same free variables in the dual formula.})  For instance the dual of an annihilation formula $xr=0$ is the corresponding divisibility formula $r|x$, that is $\exists z (rz=x)$, and (up to equivalence of formulas) {\it vice versa}.  This is a well-defined duality between the lattices ${\rm pp}^n_R$ and ${\rm pp}^n_{R^{\rm op}}$:
$D(\phi \wedge \psi) = D\phi + D\psi$; $D(\phi + \psi) = D\phi \wedge D\psi$; $D^2\phi =\phi$.
See \cite[\S 1.3]{PreNBK}.

\paragraph{Dual definable categories}
If ${\cal D}$ is a definable subcategory of ${\rm Mod}\mbox{-}R$, determined by closure of some set $\Phi$ of pp-pairs, then the ({\bf elementary}) {\bf dual definable category} ${\cal D}^{\rm d}$ is the subcategory of $R\mbox{-}{\rm Mod}$ defined by the set of dual pairs - the collection of $(D\phi, D\psi)$ such that $(\psi, \phi) \in \Phi$.

In particular, $\psi \leq _{\cal D} \phi$ iff $D\phi \leq_{{\cal D}^{\rm d}} D\psi$.

We have $({\cal D}^{\rm d})^{\rm d} ={\cal D}$.  Also $M\in {\cal D}$ iff $M^\ast \in {\cal D}^{\rm d}$ where $^\ast$ denotes any duality of the sort seen earlier in this paper, for instance $(-)^\ast = {\rm Hom}_{\mathbb Z}(-,{\mathbb Q}/{\mathbb Z})$ or, $(-)^\ast = {\rm Hom}_K(-,K)$ if $R$ is a $K$-algebra.  See \cite[\S 3.4.2]{PreNBK}.

\paragraph{Pp-types}
The {\bf pp-type of} an element $a$ in a module $M$ is the set of all pp formulas that it satisfies in $M$; similarly for $n$-tuples:  ${\rm pp}^M(\overline{a}) = \{ \phi(\overline{x}): M \models \phi(\overline{a}) \}$.  This is nothing more than the set of all projected (see ``Solution sets of pp formulas" above) $R$-linear relations satisfied by $\overline{a}$.  We say that $\overline{a}$ is a {\bf realisation} of that pp-type {\bf in} $M$.  Every set $p$ of pp formulas which is a filter, that is, upwards-closed (if $\phi \leq \psi$ and $\phi\in p$ then $\psi \in p$) and closed under intersection/conjunction ($\phi, \psi \in p$ implies $\phi \wedge \psi \in p$) occurs in this way, so we refer to such a set as a {\bf pp-type}.

A pp-type $p$ is {\bf finitely generated} if there is a pp formula $\phi \in p$ such that $p=\{ \psi: \phi \leq \psi\}$; we write $p =\langle \phi\rangle$.  If $A$ is finitely presented and $\overline{a}$ is from $A$, then ${\rm pp}^A(\overline{a})$ is finitely generated, \cite[1.2.6]{PreNBK}.

When we work in a definable subcategory ${\cal D}$, then pp-types realised in modules in ${\cal D}$ will be closed under the equivalence relation $=_{\cal D}$ and upwards closed under $\leq_{\cal D}$.  We say that a pp-type $p$ realised in ${\cal D}$ is ${\cal D}$-{\bf finitely generated} if it is generated among pp-types realised in ${\cal D}$ by a single formula, say $p = \{ \psi: \phi \leq_{\cal D} \psi\}$ and we then write $p= \langle \phi \rangle_{\cal D}$.

\paragraph{Pp-type of a module}
It makes sense to refer to the pp-type of any subset, in particular any submodule, $A$ of a module $M$:  we write $A$ as an infinite tuple $\overline{a} = (a_0, a_1, \dots, a_\lambda, \dots)_{\lambda < |A|}$ and then take the union of the pp-types in $M$ of finite sub-tuples of $\overline{a}$.  Since the ordering of $A$ is arbitrary we just write ${\rm pp}^M(A)$ for this.  If $A$ is finitely generated, by $\overline{c}$ say, then ${\rm pp}^M(A)$ is completely determined by ${\rm pp}^M(\overline{c})$ since every element of $A$ is an algebraic $R$-linear combination of elements of $\overline{c}$. 

Note that ${\rm pp}^A(A)$ is contained in ${\rm pp}^M(A)$ for any $M \supseteq A$.  So, if $A$ is a finitely presented module, generated by $\overline{c}$ with defining relations generated by the quantifier-free formula $\theta$, then ${\rm pp}^A(\overline{c})$ is minimal among pp-types (with free variables matching $\overline{c}$) which contain the formula $\theta$.

\paragraph{Morphisms}
Morphisms preserve pp formulas: if $f:M \to N$ and $M\models \phi(\overline{a})$, then $N \models \phi(f\overline{a})$ \cite[1.1.7]{PreNBK}.  Thus morphisms are non-decreasing on pp-types:  ${\rm pp}^M(\overline{a}) \subseteq {\rm pp}^N(f\overline{a})$ if $f$ is as above.  And $f$ is a {\bf pure monomorphism} iff ${\rm pp}^M(\overline{a}) = {\rm pp}^N(f\overline{a})$ for every $\overline{a}$ from $M$.

\paragraph{Free realisations of pp formulas}
A {\bf free realisation} of a pp formula $\phi$ in $n$ free variables is a finitely presented module $C$ and an $n$-tuple $\overline{c}$ from $C$ such that ${\rm pp}^C(\overline{c}) = \langle \phi\rangle$.  It follows that, if $M$ is any module and $\overline{a} \in \phi(M)$, then there is a morphism $f:C \to M$ such that $f\overline{c} = \overline{a}$ \cite[1.2.7]{PreNBK}.

Every pp formula has a free realisation \cite[1.2.14]{PreNBK}.

\paragraph{Irreducible pp-types}
A nonzero pp-type $p$ is {\bf irreducible} (or {\bf indecomposable}) if it is realised in an indecomposable pure-injective module.  That module, the {\bf hull}, denoted $H(p)$, of $p$, is unique to isomorphism (over any realisation of $p$).  Ziegler's Criterion, see \cite[4.3.49]{PreNBK}, is an often checkable criterion for this:  it says that $p$ is irreducible iff, for every $\psi_1, \psi_2$ not in $p$, there is $\phi \in p$ such that $(\psi_1 \wedge \phi) + (\psi_2 \wedge \phi) \notin p$.

\paragraph{Neg-isolated pp-types}
A pp-type $p$ is said to be {\bf neg-isolated} by a pp formula $\phi$ if it is maximal among pp-types not containing $\phi$.  Any such pp-type is \cite[\S 5.3.5, 4.3.52]{PreNBK} irreducible, so is realised in an indecomposable pure-injective.  An indecomposable pure-injective $N$ is said to be {\bf neg-isolated} if it is the hull of a neg-isolated pp-type and, in that case, every non-zero pp-type realised in it is neg-isolated, see \cite[5.3.46]{PreNBK}.  In fact neg-isolation of $N$ is equivalent, see \cite[5.3.47]{PreNBK}, to the functor $N\otimes_R -$ being the injective hull of a simple object in the functor category $(R\mbox{-}{\rm mod}, {\bf Ab})$.

All these notions relativise to any definable category ${\cal D}$, see \cite[\S 5.3.5]{PreNBK}.  In particular $N$ is {\bf neg-isolated} in ${\cal D}$, or with respect to ${\cal D}$, if it is the hull of some pp-type $p$ such that there is a pp formula $\psi$ such that $p$ is maximal, with respect to not containing $\psi$, among pp-types realised in modules in ${\cal D}$.  Also, the relevant functor categories are those associated to ${\cal D}$ and ${\cal D}^{\rm d}$ (those functor categories are Gabriel localisations of the functor categories associated to the whole module category \cite[\S 12.3]{PreNBK}).

\paragraph{Elementary cogenerators}
An {\bf elementary cogenerator} for a definable category ${\cal D}$ is a pure-injective $N\in {\cal D}$ such that every module in ${\cal D}$ is a pure submodule of a direct product of copies of $N$.  Every definable category has an elementary cogenerator \cite[5.3.52]{PreNBK} and, for $N$ to be an elementary cogenerator, it is necessary and sufficient that every neg-isolated pure-injective in ${\cal D}$ be a direct summand of $N$ \cite[5.3.50]{PreNBK}.  It is equivalent that (the localisation of) $N\otimes -$ be an injective cogenerator of the relevant functor category \cite[12.5.7]{PreNBK}, namely the Gabriel localisation of $(R\mbox{-}{\rm mod}, {\bf Ab})$ at the hereditary, finite-type torsion theory corresponding to the dual definable category ${\cal D}^{\rm d}$ (see \cite[\S 12.3]{PreNBK}).

(We have to involve modules on the other side because we are using the functor which makes a right module $M$ into a functor $M\otimes_R -$ on left modules, see \cite[\S 12.1]{PreNBK}.)

\paragraph{Rings of definable scalars}
Suppose that ${\cal D}$ is a definable subcategory.  If $\rho(x,y)$ is a pp formula with two free variables such that, on every $D\in {\cal D}$ the solution set $\rho(D)$ in $D$ is the graph of a function, necessarily additive, on $D$, then we say that $\rho$ is a {\bf definable scalar} on ${\cal D}$ (more precisely, $\rho$ defines a scalar on every module in ${\cal D}$).  Of course, if two pp formulas are equivalent on ${\cal D}$, then they define the same scalar on ${\cal D}$.  The set of maps so defined is the {\bf ring of definable scalars} for ${\cal D}$,  denoted $R_{\cal D}$.  For instance, multiplication by any $r\in R$ is such, being given by the formula $x-yr=0$, and this gives a (canonical) ring homomorphism $R \to R_{\cal D}$.  

It is easy to see that sums and compositions of pp-definable maps on ${\cal D}$ are pp-definable, so $R_{\cal D}$ is indeed a ring.  If $M$ is any module, then any pp-definable map on $M$ extends to a pp-definable map (given by the same pp formula) on the definable category ${\cal D} = \langle M \rangle$ generated by $M$ and we also write $R_M$ for $R_{\cal D}$.

For more details see \cite[Chpt.~6, \S 12.8]{PreNBK}.

Every universal localisation $R \to R'$ occurs this way (as the ring of definable scalars for ${\rm Mod}\mbox{-}R'$ canonically embedded as a definable subcategory of ${\rm Mod}\mbox{-}R$) and any ring of definable scalars $R \to R'$ can be seen as a localisation of $R$ at the level of functor categories (\cite[12.8.2]{PreNBK} makes this precise).

\paragraph{Definable closures}
If $M$ is an $R$-module, $A\subseteq M$ and $b\in M$, we say that $b$ is {\bf definable over} $A$ {\bf in} $M$ if there is $\overline{a}$ from $A$ and a formula $\chi(\overline{x},y)$ in the language of $R$-modules such that $M\models \chi(\overline{a}, b)$ and $b$ is the only solution to $\chi(\overline{a}, y)$ in $M$.  In the context, ${\cal D}_r$, that we consider, it is the case that every module $D \in {\cal D}_r$ is elementarily equivalent to $D\oplus D$ (by pp-elimination of quantifiers, see \cite[A.1.1, A.1.2]{PreNBK}) and from that it follows by \cite[2.1]{BurPre1} that a defining formula $\chi$ may be taken to be pp.  If $A \subseteq M$ is any subset, then the {\bf definable closure}, ${\rm dcl}^M(A)$, of $A$ in $M$ is the set of all elements in $M$ which are definable over $A$.  This will be a submodule of $M$ since an $R$-linear combination of elements clearly is definable over those elements.

\end{document}